\documentclass[12pt] {amsart}

\usepackage{amssymb,amsmath,amscd,amsfonts,a4,psfrag,graphicx,
latexsym,epsfig,color}
\usepackage[all]{xy}

\newtheorem{theo}{Theorem}[section]
\newtheorem{defi}[theo]{Definition}

\newtheorem{lem}[theo]{Lemma}

\newtheorem{prop}[theo]{Proposition}

\newtheorem{rem}[theo]{Remark}
\newtheorem{coro}[theo]{Corollary}

\newtheorem{exam}[theo]{Example}

\newcommand{\agot}{\ensuremath{\mathfrak{a}}}
\newcommand{\bgot}{\ensuremath{\mathfrak{b}}}

\newcommand{\kgot}{\ensuremath{\mathfrak{k}}}
\newcommand{\hgot}{\ensuremath{\mathfrak{h}}}
\newcommand{\ggot}{\ensuremath{\mathfrak{g}}}

\newcommand{\tgot}{\ensuremath{\mathfrak{t}}}
\newcommand{\ngot}{\ensuremath{\mathfrak{n}}}
\newcommand{\pgot}{\ensuremath{\mathfrak{p}}}
\newcommand{\qgot}{\ensuremath{\mathfrak{q}}}
\newcommand{\sgot}{\ensuremath{\mathfrak{s}}}

\newcommand{\Rgot}{\ensuremath{\mathfrak{R}}}

\newcommand{\Xgot}{\ensuremath{\mathfrak{X}}}

\newcommand{\Bcal}{\ensuremath{\mathcal{B}}}

\newcommand{\Ecal}{\ensuremath{\mathcal{E}}}
\newcommand{\Fcal}{\ensuremath{\mathcal{F}}}

\newcommand{\Ncal}{\ensuremath{\mathcal{N}}}

\newcommand{\Rcal}{\ensuremath{\mathcal{R}}}
\newcommand{\Ucal}{\ensuremath{\mathcal{U}}}
\newcommand{\Vcal}{\ensuremath{\mathcal{V}}}
\newcommand{\Wcal}{\ensuremath{\mathcal{W}}}
\newcommand{\Xcal}{\ensuremath{\mathcal{X}}}

\newcommand{\Qbb}{\ensuremath{\mathbb{Q}}}

\newcommand{\Cbb}{\ensuremath{\mathbb{C}}}
\newcommand{\Rbb}{\ensuremath{\mathbb{R}}}
\newcommand{\Zbb}{\ensuremath{\mathbb{Z}}}
\newcommand{\Lbb}{\ensuremath{\mathbb{L}}}
\newcommand{\Pbb}{\ensuremath{\mathbb{P}}}
\newcommand{\Jbb}{\ensuremath{\mathbb{J}}}
\newcommand{\Tbb}{\ensuremath{\mathbb{T}}}

\newcommand{\mini}{\ensuremath{\hbox{\scriptsize \rm s}}}

\newcommand{\grad}{{\ensuremath{\rm grad}}}
\newcommand{\crit}{{\ensuremath{\rm crit}}}

\newcommand{\croc}{\ensuremath{\hookrightarrow}}
\newcommand{\T}{\ensuremath{\hbox{\bf T}}}

\newcommand{\tr}{\operatorname{Tr}}

\newcommand{\infrp}{\operatorname{\hbox{\rm \tiny inf-RP}}}

\newcommand{\rp}{\operatorname{\hbox{\rm \tiny RP}}}

\newcommand{\reg}{\operatorname{\hbox{\rm \tiny reg}}}

\newcommand{\maxx}{\operatorname{\hbox{\rm \tiny max}}}

\date{December 20, 2019}

\author[Paul-Emile Paradan]{Paul-Emile Paradan}
\address{IMAG, Univ Montpellier, CNRS, Montpellier, France} 
\email{paul-emile.paradan@umontpellier.fr}

\title{Ressayre's pairs in the K\"{a}hler setting}

\begin{document}

\begin{abstract}
The aim of this article is to explain how to parameterize the equations of the facets of the Kirwan polyhedron using the notion of Ressayre's pairs.
\end{abstract}

\maketitle

\tableofcontents

%
%
%
%
%
%
%
%
%
%
%
%
%
%

\section{Introduction}

Let $M$ be a complex manifold, not necessarily compact, and let $K_\Cbb$ be a connected reductive complex Lie group acting holomorphically on $M$. We think of $K_\Cbb$ as being the complexification of 
a connected compact real Lie group $K$. We will denote by $\Jbb$ the complex structure on the tangent space $\T M$.

Let $\Omega$ be a $K$-invariant K\"{a}hler form on $M$. It means that $\Omega$ is a closed $2$-form and that the bilinear map $(v,w)\mapsto \Omega(v,\Jbb w)$ defines a Riemannian metric on $M$. We suppose that the $K$-action on $(M,\Omega)$ is Hamiltonian, so there exists a moment map $\Phi:M\to\kgot^*$ satisfying  the relations
\begin{equation}\label{eq:kostant=rel}
\iota(X_M)\Omega=d\langle\Phi,X\rangle,\quad \forall X\in\kgot.
\end{equation}
Here $X_M : m\mapsto X\cdot m:=\frac{d}{dt}\vert_{t=0}e^{tX}\cdot m$ is the vector field generated by $X\in\kgot$.

We assume that the moment map $\Phi$ is {\bf proper}. Let $T\subset K$ be a maximal torus with Lie algebra $\tgot$, and let $\tgot^*_{\geq 0}\subset\tgot^*$ 
be a closed Weyl chamber. The Convexity Theorem \cite{Atiyah82,Guillemin-Sternberg82.bis,Kirwan.84.2,L-M-T-W} tells us that the set 
$$
\Delta(\Phi):= \Phi(M)\cap \tgot^*_{\geq 0}
$$
is a convex, locally polyhedral set, called the Kirwan polyhedron. The purpose of the present paper is to explain how to parameterize the equations of the facets 
of the Kirwan polyhedron using the notion of Ressayre's pairs \cite{Ressayre10}.

\medskip

We start by introducing the notion of admissible elements. The group ${\rm Hom}(U(1),T)$ admits a natural identification with the lattice $\wedge:=\frac{1}{2\pi}\ker(\exp : \tgot\to T)$. 
A vector $\gamma\in \tgot$ is called {\em rational} if it belongs to the $\Qbb$-vector space $\tgot_\Qbb$ generated by $\wedge$. 

The stabilizer subgroup of $m\in M$ is denoted by $K_m:=\{k\in K ; k m=m\}$, and its Lie algebra by $\kgot_m$.


\begin{defi} Let us define $\dim_K(\Xcal):=\min_{m\in\Xcal}\dim(\kgot_m)$ for any subset $\Xcal\subset M$. 
A non-zero element $\gamma\in\tgot$ is called {\em admissible} if $\gamma$ is rational, and if 
$\dim_K(M^\gamma)-\dim_K(M)\in\{0,1\}$.
\end{defi}

\begin{rem}\label{rem:cas-generic-stabilizer}
$\dim_K(M)=0$ when the infinitesimal generic stabilizer of the $K$-action is reduced to $\{0\}$. In this case, a rational element $\gamma$ is admissible if 
$\dim_K(M^\gamma)\!=\!1$.
\end{rem}

For any $m\in M$, the infinitesimal action of $\kgot_\Cbb$ on $M$ defines a 
$K_m$-equivariant complex linear map 
\begin{eqnarray}\label{eq:rho-m}
\rho_m:\kgot_\Cbb & \longrightarrow & \T_m M\\
X & \longmapsto & X\cdot m .\nonumber
\end{eqnarray}

\begin{defi}
\begin{enumerate}
\item[a)] Consider the linear action $\rho: G\to {\rm GL}_\Cbb(E)$ of a compact Lie group on a complex vector space $E$. For any $(\gamma,a)\in\ggot\times\Rbb$, we define the vector subspace
$E^{\gamma=a}=\{v\in E, d\rho(\gamma)v=i av\}$.  Thus, for any $\gamma\in\ggot$, we have the decomposition
$E=E^{\gamma>0}\oplus E^{\gamma=0}\oplus E^{\gamma<0}$ where $E^{\gamma>0}=\sum_{a>0}E^{\gamma=a}$, and $E^{\gamma<0}=\sum_{a<0}E^{\gamma=a}$.
\item[b)] The real number $\tr_{\gamma}(E^{\gamma>0})$ is defined as the sum $\sum_{a>0}a\,\dim(E^{\gamma=a})$.
\item[c)] Let $\varphi:E\to F$ be a linear equivariant morphism between two $G$-module. For any $\gamma\in\ggot$, the linear map $\varphi$ specializes to a linear map
$E^{\gamma>0}\longrightarrow F^{\gamma>0}$.
\end{enumerate}
\end{defi}

\medskip

A choice of positive roots  $\Rgot^+$, induces a decomposition $\kgot_\Cbb=\ngot\oplus\tgot_\Cbb\oplus \overline{\ngot}$ 
where $\ngot=\sum_{\alpha\in\Rgot^+}(\kgot_\Cbb)_\alpha$. We denote by $B\subset K_\Cbb$ the Borel subgroup with Lie algebra 
$\bgot:=\tgot_\Cbb\oplus \ngot$.


Consider $(x,\gamma)\in M\times\tgot$ such as $x\in M^\gamma$. The $K_x$-equivariant morphism
(\ref{eq:rho-m}) induces a complex linear map 
\begin{equation}\label{eq:rho-gamma}
\rho_x^\gamma:\ngot^{\gamma>0}\longrightarrow  (\T_x M)^{\gamma>0}.
\end{equation}

\medskip

\begin{defi}\label{def:infinitesimal-B-ressayre-pair}
Let $\gamma\in\tgot$ be a non-zero element, and let $C\subset M^\gamma$ be a connected component. The data $(\gamma,C)$ is called 
an {\bf infinitesimal $B$-Ressayre's pair} if $\exists x\in C$, such as $\rho_x^\gamma$ is an isomorphism. 
If furthermore we have  $\dim_K(C)-\dim_K(M)\in\{0,1\}$, and $\lambda$ is rational, we call $(\gamma,C)$ a {\em regular infinitesimal $B$-Ressayre's pair}.
\end{defi}

\medskip

\begin{rem}
Since we work with rational vectors, for any $q\in \Qbb^{>0}$, we will identify the pairs $(\gamma,C)$ and $(q\gamma,C)$.
\end{rem}

\newpage

The first result of this article is the following theorem.

\begin{theo}\label{th:infinitesimal-ressayre-pairs} 
Let $(M,\Omega)$ be a K\"{a}hler $K$-manifold, equipped with a {\em proper} moment map $\Phi$.  
For $\xi\in\tgot^*_{\geq 0}$, the following statements are equivalent:
\begin{enumerate}
\item $\xi\in\Delta(\Phi)$.
\item For any regular infinitesimal $B$-Ressayre's pair $(\gamma,C)$, we have $\langle \xi,\gamma\rangle\geq \langle \Phi(C),\gamma\rangle$.
\item For any regular infinitesimal $B$-Ressayre's pair $(\gamma,C)$ such as $\Phi(C)\cap \tgot^*_{\geq 0}\neq\emptyset$, we have 
$\langle \xi,\gamma\rangle\geq \langle \Phi(C),\gamma\rangle$.
\end{enumerate}
\end{theo}

Vergne and Walter obtained the same kind of result when $M$ is a vector space \cite{Vergne-Walter17}.

Let us now introduce a more restrictive notion, that of $B$-Ressayre's pair. Let $\gamma\in\tgot$ be a non-zero element, 
and let $C\subset M^\gamma$ be a connected component. 
We consider the Bialynicki-Birula's complex submanifold 
\begin{equation}\label{eq:BB}
C^-:=\{m\in M,  \lim_{t\to\infty} \exp(-it\gamma) m\ \in C\}.
\end{equation}
We see that for any $x\in C$, $(\T_x M)^{\gamma\leq 0}=\T_x C^-$. 
Consider now the parabolic subgroup $P_\gamma\subset K_\Cbb$ defined by
\begin{equation}\label{eq:P-gamma}
P_\gamma=\{g\in K_\Cbb, \lim_{t\to\infty}\exp(-it\gamma)g\exp(it\gamma)\ {\rm exists}\}.
\end{equation}
Note that $C^-$ is invariant under the action of $P_\gamma$, hence we can  consider the complex manifold 
$B\times_{B\cap P_\gamma} C^-$ and the holomorphic map 
$$
{\rm q}_\gamma: B\times_{B\cap P_\gamma} C^-\to M
$$
that sends $[b,x]$ to $bx$. We immediately see that for any $x\in C$, the tangent map 
$\T{\rm q}_\gamma\vert_x$ is an isomorphism if and only if $\rho_x^\gamma$ is an isomorphism.

\begin{defi}\label{def:B-ressayre-pair}
Let $\gamma\in\tgot$ be a non-zero element, and let $C\subset M^\gamma$ be a connected component. The data $(\gamma,C)$ is called a {\bf $B$-Ressayre's pair} if the following conditions hold
\begin{itemize}
\item The holomorphic map ${\rm q}_{\gamma}: B\times_{B\cap P_{\gamma}}C^-\to M$ is dominant\footnote{The image of ${\rm q}_{\gamma}$ contains a dense open subset.}.
\item There exists a $B\cap P_{\gamma}$-invariant, open and dense subset $U\subset C^-$, intersecting $C$, 
so that ${\rm q}_{\gamma}$ defines a diffeomorphism $B\times_{B\cap P_{\gamma}}U\simeq B U$.
\end{itemize}
If furthermore we have  $\dim_K(C)-\dim_K(M)\in\{0,1\}$, and $\lambda$ is rational, we call $(\gamma,C)$ a regular $B$-Ressayre's pair.
\end{defi}

The second result of this article is the following theorem

\medskip

\begin{theo}\label{th:ressayre-pairs} 
Let $(M,\Omega)$ be a K\"{a}hler $K$-manifold, equipped with a {\em proper} moment map $\Phi$. 
For $\xi\in\tgot^*_{\geq 0}$, the following statements are equivalent:
\begin{enumerate}
\item $\xi\in\Delta(\Phi)$.
\item For any regular $B$-Ressayre's pair $(\gamma,C)$, we have $\langle \xi,\gamma\rangle\geq \langle \Phi(C),\gamma\rangle$.
\item For any regular $B$-Ressayre's pair $(\gamma,C)$ such as $\Phi(C)\cap \tgot^*_{\geq 0}\neq\emptyset$, we have $\langle \xi,\gamma\rangle\geq \langle \Phi(C),\gamma\rangle$.
\end{enumerate}
\end{theo}

\medskip

In \S \ref{sec:RP=infRP}, we will explain under which circumstances the notions of  {\em infinitesimal $B$-Ressayre's pair} and {\em $B$-Ressayre's pair} coincide. Let us denote by 
$\tgot^*_{> 0}$ the interior of the Weyl chamber $\tgot^*_{\geq 0}$.

\begin{theo}\label{th:3} 
Let  $(\gamma,C)$ be a infinitesimal $B$-Ressayre's pair of a complex $K_\Cbb$-variety $M$. Let $(\Omega,\Phi)$ be a K\"{a}hler-Hamiltonian structure on $M$ such as 
$\Phi(C)\cap \tgot^*_{> 0}\neq\emptyset$. Then $(\gamma,C)$ is a $B$-Ressayre's pair. 
\end{theo}

\medskip

\subsection*{Acknowledgements} This article, which is strongly influenced by the work of Nicolas Ressayre, takes up several of his ideas in the differential setting. 
I wish to thank Mich\`ele Vergne for our discussions on this subject and for pointing me out an error in a preliminary work where I was trying to develop the notion of Ressayre's pairs  
in symplectic geometry. Thanks to her, I understood that the K\"{a}hlerian framework was the right one to carry out this project.

\medskip

\subsection*{Notations} Throughout the paper, by {\em K\"{a}hler Hamiltonian $K$-manifold} we mean a complex manifold $M$ equipped with an action of $K_\Cbb$
 by holomorphic transformations and a $K$-invariant K\"{a}hler $2$-form $\Omega$: we suppose furthermore that the $K$-action on 
 $(M,\Omega)$ is Hamiltonian.

\section{Ressayre's pairs}

\subsection{The Bialynicki-Birula's complex submanifolds}\label{sec:BB}

Let $M$ be a K\"{a}hler Hamiltonian $K$-manifold.

Let us consider an element $\gamma\in\kgot$ and a connected component
$C$ of the complex submanifold $M^\gamma:=\{m\in M, \gamma\cdot m=0\}$. As in the introduction, we define the 
subset $C^-:=\{m\in M,  \lim_{t\to\infty} \exp(-it\gamma)m\ \in C\}$ and the projection $p_C:C^-\to C$ that sends $m\in C^-$ to 
$\lim_{t\to\infty} \exp(-it\gamma) m\ \in C$.

\begin{prop}
$C^-$ is a locally closed complex submanifold of $M$, and the projection $p_C:C^-\to C$ is an holomorphic map.
\end{prop}

\begin{proof}
Let $\Tbb\subset K$ be the torus equal to the closure of $\exp(\Rbb \gamma)$, and let $\Tbb_\Cbb\subset K_\Cbb$ be its complexification. 
Let $m\in C^-$ and let $x=p_C(m)\in C\subset M^{\Tbb_\Cbb}$. A theorem of Koras \cite{Koras-86, Sjamaar-95} tells us that $\Tbb_\Cbb$-action 
can be linearized near $x$. In other words, there exists 
an holomorphic $\Tbb_\Cbb$-equivariant diffeomorphism $\varphi :\Ucal\to \Vcal$ where $\Ucal$ is a $\Tbb_\Cbb$-invariant open neighborhood of $0$ in $\T_x M$ and 
$\Vcal$ is a $\Tbb_\Cbb$-invariant open neighborhood of $x$ in $M$.

By definition $\exp(-it\gamma)m$ tends to $x$, when $t$ goes to infinity. Then $\exp(-it\gamma) m\in \Vcal$ when $t$ is large enough. But since $\Vcal$ is 
$\Tbb_\Cbb$-invariant, we have then $m\in\Vcal$. Throught $\varphi$, we see that  $C^-\cap \Vcal$ is diffeomorphic to $(\T_x M)^{\gamma \leq 0}\cap \Ucal$, and that the map 
$p_C: C^-\cap \Vcal\to C\cap \Vcal$ corresponds to the projection $(\T_x M)^{\gamma \leq 0}\cap \Ucal\to (\T_x M)^{\gamma = 0}\cap \Ucal$. It should be noted here that 
$(\T_x M)^{\gamma \leq 0}\cap \Ucal= (\T_x M)^{\gamma = 0}\cap \Ucal \oplus (\T_x M)^{\gamma< 0}$.
\end{proof}

When $M$ is a projective variety, Bialynicki-Birula shows that $C^-$ is a subvariety Zariski dense in its closure \cite{Bialynicki-Birula}. We will not need this result here.

We finish this section with a remark that we will need later. Consider an holomorphic line bundle $\Lbb\to C^-$ that is equivariant relatively to the action of the torus $\Tbb$. 
We consider its restriction $\Lbb\vert_C$. If $\theta \in H^0(C^-,\Lbb)^\Tbb$ is an equivariant holomorphic section, we can consider its restriction 
$\theta\vert_C$.

\begin{lem}\label{lem:action-fibre-L}
Let $\theta \in H^0(C^-,\Lbb)^\Tbb$.
\begin{enumerate}
\item If the $\Tbb$-action on $\Lbb\vert_C$ is not trivial, then $\theta\vert_C=0$.
\item If the $\Tbb$-action on $\Lbb\vert_C$ is trivial, then 
$$
\{m\in C^-;\,\theta(m)\neq 0\}=p^{-1}_C\left(\{x\in C;\,\theta\vert_C(x)\neq 0\}\right).
$$
\end{enumerate}
\end{lem}

\begin{proof}
The first point is obvious to see.

Let $m\in C^-$ and $x=p_C(m)$. In the previous proposition, we have explained that a $\Tbb$-invariant neighborhood $\Wcal$ of $m$ in $C^-$ admits 
an holomorphic diffeomorphism with $\widetilde{\Wcal}:=U_0\times (\T_x M)^{\gamma< 0}$ where $U_0$ is an open neighborhood of $0$ in $(\T_x M)^{\gamma = 0}$. 
In this model, the projection $p_C$ is $(v_0,v_{<0})\in \widetilde{\Wcal}\mapsto v_0$.

For any $\Tbb$-equivariant holomorphic line bundle $\Lbb$ on $C^-$, we have a $\Tbb$-equivariant isomorphism between 
$\Lbb\vert_\Wcal$ and $U_0\times (\T_x M)^{\gamma< 0}\times \Lbb_x$. Then, a section $\theta \in H^0(C^-,\Lbb)^\Tbb$ defines a $\Tbb$-equivariant map 
$\theta\vert_\Wcal:U_0\times (\T_x M)^{\gamma< 0}\to\Lbb_x$. If the $\Tbb$-action on $\Lbb_x$ is trivial, we see that 
$\theta\vert_\Wcal(v_o,v_{<0})=\theta\vert_\Wcal(v_o,0)$, because $t \mapsto \theta\vert_\Wcal(v_o,e^{-it\gamma} v_{<0})$ is a constant map.

\end{proof}

\subsection{$K_\Cbb$-Ressayre's pairs}\label{sec:KRP}

Let $N$ be a K\"{a}hler Hamiltonian $K$-manifold, and let $\gamma\in\tgot$ be a non-zero element. When $x\in N^\gamma$, the $K_x$-equivariant morphism
(\ref{eq:rho-m}) induces a complex linear map 
\begin{equation}\label{eq:rho-gamma-K}
(\kgot_\Cbb)^{\gamma>0}\longrightarrow  (\T_x N)^{\gamma>0}.
\end{equation}

\begin{defi}\label{def:infinitesimal-K-ressayre-pair}
Let $C_N\subset N^\gamma$ be a connected component. The data $(\gamma,C_N)$ is called 
an {\bf infinitesimal $K_\Cbb$-Ressayre's pair} if $\exists x\in C_N$, such as (\ref{eq:rho-gamma-K}) is an isomorphism. If furthermore we have 
$\dim_K(C)-\dim_K(N)\in\{0,1\}$, and $\lambda$ is rational, we call $(\gamma,C_N)$ a {\em regular infinitesimal $K_\Cbb$-Ressayre's pair}.
\end{defi}

Consider the Bialynicki-Birula's complex submanifold $C^-_N$ defined by (\ref{eq:BB}). We denote by $p:C_N^-\to C_N$ the projection. As  $C^-_N$ is invariant 
under the action of the parabolic subgroup $P_\gamma$, we may consider the holomorphic map 
$$
{\rm \pi}_\gamma: K_\Cbb\times_{P_\gamma} C^-_N\to N
$$
that sends $[g,x]$ to $gx$. Let $(C^-_N)_{\reg}$ be the open subset formed by the point $n\in C^-_N$ such as the tangent map 
$\T{\rm \pi}_\gamma\vert_{[e, n]}$ is an isomorphism. We immediatly see that  $x\in C_N\cap (C^-_N)_{\reg}$ if and only 
if (\ref{eq:rho-gamma-K}) is an isomorphism. 

\begin{lem}\label{lem:C-N-reg}
Let $(\gamma,C_N)$ be an  infinitesimal $K_\Cbb$-Ressayre's pair. Then $(C^-_N)_{\reg}$ is a dense, $P_\gamma$-invariant, open subset of $C^-_N$ such as 
$(C^-_N)_{\reg}=p^{-1}(C_N\cap (C^-_N)_{\reg})$.
\end{lem}

\begin{proof}
The tangent map $\T\pi_\gamma\vert_{[e,n]}$ is an isomorphism if 
$$
\begin{cases}
\kgot_\Cbb\cdot n + \T_n C^-_N=\T_n M\\
\kgot_\Cbb\cdot n\cap \T_n C^-_N \simeq \pgot_\gamma,
\end{cases}
$$
where $\pgot_\gamma=(\kgot_\Cbb)^{\gamma\leq 0}$ is the Lie algebra of $P_\gamma$.

As $(\gamma,C_N)$ is an infinitesimal $K_\Cbb$-Ressayre's pair, the previous relations are satisfied on some points of $C_N$.
Then, the rank $r$ of the holomorphic bundle $\mathbb{E}:=\T M\vert_{C_N^-}/\T C^-_N$ is equal to the dimension of $\kgot_\Cbb/\pgot_\gamma$.
We consider now the $P_\gamma$-holomorphic line bundle $\Lbb\to C^-_N$ defined by 
$\Lbb:=\hom \left(\wedge^r(\kgot_\Cbb/\pgot_\gamma), \wedge^r\mathbb{E}\right)$. 
We have a canonical $P_\gamma$-equivariant section $\theta : C^-_N\to \Lbb$ defined  by
$$
\theta(n) : \overline{X_1} \wedge \cdots\wedge  \overline{X_r}\longrightarrow  \overline{X_1\cdot n} \wedge \cdots\wedge  \overline{X_r\cdot n}.
$$

We notice that  $(C^-_N)_{\reg}=\{n; \theta(n)\neq 0\}$ and that the torus $\Tbb=\overline{\exp(\Rbb\gamma)}$ acts trivially on $\Lbb\vert_{C_N}$. Thanks to 
Lemma \ref{lem:action-fibre-L}, we can conclude that $n\in (C^-_N)_{\reg}$ if and only if $p(n)\in (C^-_N)_{\reg}$.
\end{proof}

Let us now introduce a more restrictive notion, that of $K_\Cbb$-Ressayre's pair.

\begin{defi}\label{def:K-ressayre-pair}
The data $(\gamma,C_N)$ is called a {\bf $K_\Cbb$-Ressayre's pair} if the following conditions hold
\begin{itemize}
\item The holomorphic map ${\rm \pi}_{\gamma}$ is dominant.
\item There exists a $P_{\gamma}$-invariant, open and dense subset $U\subset C^-_N$, intersecting $C_N$, 
so that ${\rm \pi}_{\gamma}$ defines a diffeomorphism $K_\Cbb\times_{P_\gamma} U\simeq K_\Cbb U$.
\end{itemize}
If furthermore we have  $\dim_K(C)-\dim_K(N)\in\{0,1\}$, and $\lambda$ is rational, we call $(\gamma,C_N)$ a regular $K_\Cbb$-Ressayre's pair.
\end{defi}

\subsection{$K_\Cbb$-Ressayre's pairs versus $B$-Ressayre's pairs}\label{sec:KandB}

Let $M$ be a K\"{a}hler Hamiltonian $K$-manifold. We associate to it the complex manifold $N:=M\times K_\Cbb/B$. To a
connected component $C$ of $M^\gamma$ we associate the connected component 
$$
C_N:=C\times K^\gamma_\Cbb/K^\gamma_\Cbb\cap B
$$ 
of $N^\gamma$.

\begin{prop}\label{prop:K-B-RP}
\begin{enumerate}
\item $(\gamma,C)$ is a $B$-Ressayre's pair on $M$ if and only if $(\gamma,C_N)$ is a $K_\Cbb$-Ressayre's pair on $N$.
\item $(\gamma,C)$ is an infinitesimal $B$-Ressayre's pair on $M$ if and only if $(\gamma,C_N)$ is an infinitesimal $K_\Cbb$-Ressayre's pair on $N$.
\end{enumerate}
\end{prop}
\begin{proof}
Let us prove the first point. The Bialynicki-Birula's complex manifold $C_N^-$ is equal to $C^-\times P_\gamma/P_\gamma\cap B$. We have to compare the maps  
${\rm q}_\gamma: B\times_{B\cap P_\gamma} C^-\to M$, and $\pi_\gamma : K_\Cbb\times_{P_\gamma} C_N^-\to N$. Consider the  canonical isomorphism 
$C_N=C^-\times P_\gamma/P_\gamma\cap B\simeq P_\gamma\times_{P_\gamma\cap B} C^-$. More generaly, for any $P_\gamma$-invariant open subset 
$\Ucal_N\subset C^-_N$ we have an isomorphism $\Ucal_N\simeq P_\gamma\times_{P_\gamma\cap B} \Ucal$ where $\Ucal$ is the open $P_\gamma\cap B$-invariant subset of 
$C^-$ defined by the relation $\Ucal:=\{x\in C^-; (x,[e])\in \Ucal_N\}$. Note that $\Ucal_N$ intersects $C_N$ if and only if $\Ucal$ intersects $C$. 

Now, we notice that ${\rm q}_\gamma$ defines a diffeomorphism $B\times_{B\cap P_\gamma} \Ucal\simeq B\Ucal$ if and only if 
$\pi_\gamma$ defines a diffeomorphism $K_\Cbb\times_{P_\gamma} \Ucal_N\simeq K_\Cbb \Ucal_N$. It can be seen easily through the commutative diagram
$$
\xymatrixcolsep{5pc}\xymatrix{
K_\Cbb\times_{P_\gamma} \Ucal_N\ar[d]^{\pi_\gamma} \ar[r]^{\sim} & K_\Cbb\times_{B} (B\times_{B\cap P_\gamma} \Ucal)\ar[d]^{{\rm q}_\gamma} \\
N \ar[r]^{\sim}          & K_\Cbb\times_B M.
}
$$

The second point is immediate since for any $x\in C$, the tangent map $\T{\rm q}_\gamma\vert_x$ is an isomorphism if and only if 
$\T{\rm \pi}_\gamma\vert_{(x,[e])}$ is an isomorphism

\end{proof}

\subsection{Ressayre's pairs and semi-stable points  I}\label{sec:RP-semi-stable-1}

Let $N$ be a K\"{a}hler Hamiltonian $K$-manifold with proper moment map $\Phi_N$. In this section we suppose that 
$\Phi_N^{-1}(0)\neq \emptyset$, and we consider the subset of analytical semi-stable points:
$$
N^{ss}=\{n\in N; \overline{K_\Cbb\, n}\cap \Phi_N^{-1}(0)\neq\emptyset\}.
$$
In \S \ref{sec:square-phi}, we will explain the well-known fact that\footnote{The set $N^{ss}$ is denoted by $N_0$ in \S \ref{sec:square-phi}.}  $N^{ss}$ is a dense open subset of $N$.

Let $(\gamma,x)\in\kgot\times N$, such as the limit $x_\gamma:=\lim_{t\to\infty}e^{-it\gamma} x$ exists in $N$.  
Recall that $t\geq 0\mapsto e^{-it\gamma} x$ corresponds to the gradient flow of the Morse-Bott fonction $-\langle\Phi,\gamma\rangle$. 

The result of the next proposition is classical in the projective case (see \cite{Ressayre10}, Lemma 2).

\begin{prop}\label{prop:N-ss-gamma}
Let $(\gamma,x)\in\kgot\times N$, such as the limit $x_\gamma:=\lim_{t\to\infty}e^{-it\gamma} x$ exists. 
Then, the following hold
\begin{enumerate}
\item If $x\in N^{ss}$, then $\langle \Phi_N(x_\gamma),\gamma\rangle\leq 0$.
\item If $x\in N^{ss}$ and $\langle \Phi_N(x_\gamma),\gamma\rangle= 0$, then $x_\gamma\in N^{ss}$.
\end{enumerate}
\end{prop}

\begin{proof}Let $P_\gamma\subset K_\Cbb$ be the parabolic subgroup associated to $\gamma$ (see (\ref{eq:P-gamma})). Since $K_\Cbb=KP_\gamma$, the fact 
that $x\in N^{ss}$ means that $\overline{P_\gamma x}\cap\Phi_N^{-1}(0)\neq \emptyset$: in other words $\min_{n\in P_\gamma x}\|\Phi_N(n)\|=0$.

Consider now the function $t\geq 0\mapsto \langle \Phi_N(e^{-it\gamma} n),\gamma\rangle$ attached to $n\in P_\gamma x$.  Since 
$\frac{d}{dt}\langle\Phi_N(e^{-it\gamma}n),\gamma\rangle= -\|\gamma_M\|^2(e^{-it\gamma}n)\leq 0$, 
we have 
\begin{equation}\label{eq:N-ss-gamma}
\langle\Phi_N(n),\gamma\rangle \geq \langle\Phi_N(e^{-it\gamma}n),\gamma\rangle,\qquad \forall t\geq 0.
\end{equation} 

Let $C_N\subset N$ be the connected component of $N^\gamma$ containing $x_\gamma$: it is a complex submanifold of $N$, stable by the $K^\gamma_\Cbb$-action. 
Moreover,  the function $n \mapsto \langle \Phi_N(n),\gamma\rangle$ is constant when restricted to $C_N$, equal to $\langle \Phi_N(x_\gamma),\gamma\rangle$. 
Let's take $p\in P_\gamma$ and $n=px$. Then, the limit $\lim_{t\to\infty}e^{-it\gamma} n$ is equal to $k x_\gamma\in C_N$ where 
$k=\lim_{t\to\infty}e^{-it\gamma} p \, e^{it\gamma}\in K^\gamma_\Cbb$. If we take the limit  in (\ref{eq:N-ss-gamma}) as $t\to\infty$, we obtain 
$\langle\Phi_N(n),\gamma\rangle\geq \langle \Phi_N(x_\gamma),\gamma\rangle$ for 
any $n\in P_\gamma x$. Since $\min_{n\in P_\gamma x}\|\Phi_N(n)\|=0$, we can conclude that 
$0\geq \min_{n\in P_\gamma x}\langle\Phi_N(n),\gamma\rangle\geq \langle \Phi_N(x_\gamma),\gamma\rangle$.

The second point is proved in the Appendix.
\end{proof}

\section{The Kirwan-Ness stratification}

\subsection{The square of the moment map}\label{sec:square-phi}

Let us now choose a rational invariant inner product on $\kgot^*$. By rational we mean that for a maximal torus $T\subset K$ with Lie algebra $\tgot$, the inner product takes integral values on the lattice $\wedge:=\frac{1}{2\pi}\ker(\exp:\tgot\to T)$. Let us denote by $\wedge^*\subset \tgot^*$ the dual lattice : $\wedge^*=\hom(\wedge,\Zbb)$. 
We associate to the lattices $\wedge$ and $\wedge^*$ the $\Qbb$-vector space $\tgot_\Qbb$ and $\tgot^*_\Qbb$ generated by them: the vectors belonging to 
$\tgot_\Qbb$ and $\tgot^*_\Qbb$ are designed as rational.

The invariant scalar product on $\kgot$ induces an identification $\kgot^*\simeq \tgot,\xi\mapsto\tilde{\xi} $ such as $\tgot_\Qbb\simeq\tgot^*_\Qbb$. 
To simplify our notation, we will not distinguish between $\xi$ and $\tilde{\xi}$: for example we write $M^{\lambda}$ for the submanifold fixed by 
the subgroup generated by $\tilde{\lambda}$, and we denote by $K^\gamma\subset K$ the subgroup that leaves $\tilde{\lambda}$ invariant.

Let $M$ be a K\"{a}hler Hamiltonian $K$-manifold with proper moment map $\Phi:M\to\kgot^*$. Let  
$$
f:=\frac{1}{2}(\Phi,\Phi):M\longrightarrow \Rbb
$$
denote the norm-square of the moment map. Notice that $f$ is a proper function on $M$.

\begin{defi}
The Kirwan vector field on $M$ is defined by the relation
$$
\kappa_\Phi(m)=\Phi(m)\cdot m,\quad \forall m\in M.
$$
\end{defi}

We consider the gradient $\grad(f)$ of the function $f$ relatively to the Riemannian metric $\Omega(-,\Jbb-)$. We recall the following well-known facts \cite{W11}.

\begin{prop}
\begin{enumerate}
\item The set of critical points of the function $f$ is $\crit(f)=\{\kappa_\Phi=0\}$.
\item The gradient of $f$ is $\grad(f)=\Jbb(\kappa_\Phi)$.
\item We have the decomposition
$\Phi(\crit(f))=\bigcup_{\lambda\in \Bcal_\Phi} K\lambda$ 
where the set $\Bcal_\Phi\subset \tgot^*_{\geq 0}$ is discrete. $\Bcal_\Phi$ is called the set of {\rm types} of $M$. 
\item We have the decomposition $\crit(f)=\bigcup_{\lambda\in \Bcal_\Phi}Z_\lambda$ 
where $Z_\lambda=\crit(\phi)\cap\Phi^{-1}(K\lambda)$ is equal to $K(M^{\lambda}\cap \Phi^{-1}(\lambda))$.
\end{enumerate}
\end{prop}

Let $\varphi_t:M\to M$ be the flow of $-\grad(f)$; since $f$ is proper, $\varphi_t$ exists for all times
$t\in [0,\infty[$, and according to a result of Duistermaat \cite{Lerman05} we know that any trajectory of $\varphi_t$ has a limit when $t\to\infty$. For any 
$m\in M$, let us denote $m_\infty:=\lim_{t\to\infty} \varphi_t(m)$.

The construction of the Kirwan-Ness stratification goes as follows. For each $\lambda\in\Bcal_\Phi$, let $M_\lambda$ denote the set of points of $M$ flowing to $Z_\lambda$,
$$
M_\lambda:=\{m\in M; m_\infty\in Z_\lambda\}.
$$
From its very definition, the set $M_\lambda$ is contained in $\{m\in M, \phi(m)\geq \frac{1}{2}\|\lambda\|^2\}$.

The Kirwan-Ness stratification is the decomposition \cite{Kirwan.84.1}, \cite{Ness84}: 
$$
M=\bigcup_{\lambda\in \Bcal_\Phi} M_\lambda.
$$

When $0$ belongs to the image of $\Phi$, the strata $M_0$ corresponds to the open subset of analytical semi-stable points:
$M_0=\{m\in M; \overline{K_\Cbb\, m}\cap \Phi^{-1}(0)\neq\emptyset\}$.

Let us now explain the geometry of $M_\lambda$ for a non-zero type $\lambda$. Let $C_\lambda$ be the union of the connected components of $M^{\lambda}$ intersecting 
$\Phi^{-1}(\lambda)$. Then $C_\lambda$ is a K\"{a}hler Hamiltonian $K^\lambda$-manifold with proper moment map $\Phi_\lambda:=\Phi\vert_{C_\lambda}-\lambda$.

The Bialynicki-Birula's complex submanifold 
$$
C^-_\lambda:=\{m\in M,  \lim_{t\to\infty} \exp(-it\lambda)\cdot m\ \in C_\lambda\}
$$
corresponds to the set of points of $M$ flowing to $C_\lambda$ under the flow $\varphi_{\lambda,t}$ of $-\grad\langle\Phi,\lambda\rangle$, as $t\to\infty$. The limit of the flow defines 
a projection $C^-_\lambda\to C_\lambda$. Notice that $C^-_\lambda$ is invariant under the action of the parabolic subgroup $P_\lambda$. 

Consider now the Kirwan-Ness stratification of the K\"{a}hler Hamiltonian $K_\lambda$-manifold $C_\lambda$. Let $C_{\lambda,0}$ be the open strata of 
$C_\lambda$ corresponding to the $0$-type:
$$
C_{\lambda,0}=\{x\in C_{\lambda};\ \overline{K^\lambda_\Cbb\, x}\cap \Phi^{-1}_\lambda(0)\neq \emptyset\}.
$$
Let $C^-_{\lambda,0}$ denotes the inverse image of $C_{\lambda,0}$ in $C^-_\lambda$.

\begin{theo}[Kirwan \cite{Kirwan.84.1}]\label{Kirwan-stratification}
Let $M$ be a K\"{a}hler Hamiltonian $K$-manifold with proper
moment map $\Phi:M\to\kgot^*$. For each non zero type $\lambda$, $M_\lambda$ is a $K_\Cbb$-invariant complex submanifold, and 
$K_\Cbb\times_{P_\lambda}C^-_{\lambda,0}\rightarrow M_\lambda$, $[g,z] \mapsto  g\cdot z$ 
is an isomorphism of complex  $K_\Cbb$-manifolds.
\end{theo}

\begin{rem}
Kirwan gave a proof when $M$ is a {\em compact} K\"{a}hler Hamiltonian $K$-manifold. When 
$M$ is non-compact but the moment map is proper, a proof is given in \cite{HSS08} (see also \cite{W11}).
\end{rem}

Theorem \ref{Kirwan-stratification} gives a useful corollary.

\begin{coro}\label{coro:coro-stratification}
For any $m\in M$, we have $\|\Phi(m_\infty)\|=\inf_{x\in K_\Cbb m}\|\Phi(x)\|$.
\end{coro}

\medskip

\subsection{The minimal type}

Let $\Delta(\Phi)\subset \tgot^*_{\geq 0}$ be the Kirwan polyhedron, and let $\Bcal_\Phi$ be the set of types of our 
 K\"{a}hler Hamiltonian $K$-manifold $M$ with proper moment map $\Phi$. Let $\lambda_{\mini}$ be the orthogonal projection of 
 $0$ on the closed convex polyhedron $\Delta(\Phi)$. 

We start with the following basic facts.

\begin{lem}
\begin{enumerate}
\item[a)] $\lambda_{\mini}$ is the unique element of $\Bcal_\Phi$ with minimal norm.
\item[b)] The critical subset $Z_{\lambda_{\mini}}$ is equal to $\Phi^{-1}(K\lambda_{\mini})$.
\item[c)] The submanifold $C_{\lambda_{\mini}}$ is the connected component of $M^{\lambda_{\mini}}$ containing $\Phi^{-1}(\lambda_{\mini})$.
\item[d)] The strata $M_{\lambda_{\mini}}$ is connected.
\end{enumerate}
\end{lem}

\begin{proof}
Any element $m\in\Phi^{-1}(K\lambda_{\mini})$ is a critical point of $f$, as $\frac{1}{2}\|\lambda_{\mini}\|^2=f(m)=\min_M f$. 
It proves that $\lambda_{\mini}$ is an element of $\Bcal_\Phi$, that  $\Phi^{-1}(\lambda_{\mini})\subset M^{\tilde{\lambda}_{\mini}}$ and that 
$Z_{\lambda_{\mini}}=\Phi^{-1}(K\lambda_{\mini})$. Moreover since $\Bcal_\Phi\subset\Delta(\Phi)$, $\lambda_{\mini}$ is the unique element of $\Bcal_\Phi$ with minimal norm.

Since $\Phi^{-1}(\lambda_{\mini})$ is connected, there exists a unique connected component 
$C_{\lambda_{\mini}}$ of $M^{\lambda_{\mini}}$ containing $\Phi^{-1}(\lambda_{\mini})$. Since $\Phi^{-1}(K\lambda_{\mini})$ is connected, the strata 
$M_{\lambda_{\mini}}:=\{m\in M; \lim_{t\to\infty}\varphi_t(m)\in\Phi^{-1}(K\lambda_{\mini})\}$ is also connected.
\end{proof}

It should be noted that $\lambda_{\mini}$ is not rational in general, as shown in the following example.

\begin{exam}
Consider a regular element $\xi\in\tgot^*_{\geq 0}$, and the K\"{a}hler Hamiltonian $K$-manifold $M:=K(2\xi) \times \overline{K\xi}$. Here the minimal type is 
$\lambda_{\mini}=\xi$. So the minimal type of $M$ is not rational if $\xi\notin\tgot^*_\Qbb$.
\end{exam}

We recall the following fundamental fact.

\begin{prop}\label{prop:M-lambda-open}
The strata $M_\lambda$ has a non-empty interior if and only if $\lambda=\lambda_{\mini}$.
\end{prop}

\begin{coro}\label{prop:strate-0-ouverte}
The strata $M_{\lambda_{\mini}}$ is an open, dense and $K_\Cbb$-invariant subset of $M$.
\end{coro}

\begin{proof}
For any type $\lambda$, the strata $M_\lambda$ is contained in the closed subset $M(\|\lambda\|):=\{m\in M, \phi(m)\geq \frac{1}{2}\|\lambda\|^2\}$. If we take $r=\inf\{\|\lambda\|, \lambda\in\Bcal_\Phi-\{\lambda_{\mini}\}\}>\|\lambda_{\mini}\|$, we see that the union 
$\cup_{\lambda\neq\lambda_{\mini}} M_\lambda$ is contained in $M(r)$. Hence the non-empty open subset $M-M(r)$ is contained in the strata $M_{\lambda_{\mini}}$: thus $M_{\lambda_{\mini}}$ is an open subset of $M$.

Consider a strata $M_\lambda$ with a non-empty interior. This means that a connected component $M^o_\lambda$ of the submanifold $M_\lambda$ is an open subset of $M$. 
By definition, $M^o_\lambda$ intersects $\Phi^{-1}(\lambda)$: so let's take $x_o\in \Phi^{-1}(\lambda)\cap M^o_\lambda$. We have 
$f\geq \frac{1}{2}\|\lambda\|^2$ on $M^o_\lambda$ and $f(x_o)= \frac{1}{2}\|\lambda\|^2$: hence $x_o$ is a local minimum of the function $f$.

We finish the proof of Proposition \ref{prop:M-lambda-open} with the

\begin{lem}
Let $x_o\in M$ be a local minimum of $f$. Then $f(x_o)\in K\lambda_{\mini}$.
\end{lem}
\begin{proof}
For any $K$-invariant open subset $U\subset M$, we consider $\Delta(U)=\Phi(U)\cap\tgot^*_{\geq 0}$ which is a subset of the Kirwan polyhedron $\Delta(\Phi)$. 
Up to a change by the $K$-action, we can assume that $\mu:=\Phi(x_o)$ belongs to the Weyl chamber. Let us denote by $B_\mu(r)\subset \tgot^*$ the open ball centered at $\mu$ 
and of radius $r>0$.

The local convexity theorem of Sjamaar \cite{Sjamaar-98} (Theorems 6.5 and 6.7) tells us that for any $K$-invariant open subset $U$ containing $x_o$, there exists $r>0$ such as 
$$
\Delta(\Phi)\cap B_\mu(r)=\Delta(U)\cap B_\mu(r).
$$

If $x_o\in M$ is a local minimum of $f$, there exists a $K$-invariant open neighborhood $U_o$ of $x_o$ so that $\|\xi\|\geq \|\mu\|$ for any $\xi\in \Delta(U_o)$. Thus for $r_o>0$ small enough, $\Delta(\Phi)\cap B_\mu(r_o)$ is contained in $\{\xi\in\tgot^*, \|\xi\|\geq \|\mu\|\}$. Since  $\Delta(\Phi)\cap B_\mu(r_o)$ is convex, it implies that 
\begin{equation}\label{eq:delta-phi-r}
(\xi,\mu)\geq \|\mu\|^2,\quad \forall \xi\in\Delta(\Phi)\cap B_\mu(r_o).
\end{equation}
Using now the convexity of $\Delta(\Phi)$, we see that (\ref{eq:delta-phi-r}) forces $\Delta(\Phi)$ to be contained in the half-space $(\xi,\mu)\geq \|\mu\|^2$: hence $\mu$ is the smallest 
element of $\Delta(\Phi)$.
\end{proof}
\end{proof}

Let us finish this section by considering the particular situation where the strata $M_0$ is empty. In other words, we suppose that the minimal type 
$\lambda_{\mini}$ is non-zero. Let $C_{\lambda_{\mini}}$ be the connected component of $M^{\lambda_{\mini}}$ containing $\Phi^{-1}(\lambda_{\mini})$. 
We consider the Bialynicki-Birula's complex submanifold  $C^-_{\lambda_{\mini}}$ and the holomorphic map 
$\pi_{\lambda_{\mini}}: K_\Cbb\times_{P_{\lambda_{\mini}}}C^-_{\lambda_{\mini}}\to M$.

Theorem \ref{Kirwan-stratification} gives us the following important fact.
\begin{prop}\label{prop:minimal-type-non-zero}
When $\lambda_{\mini}\neq 0$, 
\begin{itemize}
\item The map $\pi_{\lambda_{\mini}}$ is dominant.
\item There exists a $P_{\lambda_{\mini}}$-invariant open and dense subset $U\subset C^-_{\lambda_{\mini}}$, intersecting $C_{\lambda_{\mini}}$, 
so that $\pi_{\lambda_{\mini}}$ defines a diffeomorphism $K_\Cbb\times_{P_{\lambda_{\mini}}}U\simeq K_\Cbb U$.
\end{itemize}
In other words, the data $(\lambda_{\mini}, C_{\lambda_{\mini}})$ is a {\em $K_\Cbb$-Ressayre's pairs} on $M$.
\end{prop}

\subsection{Construction of $B$-Ressayre's pairs}\label{sec:construction-RP}
 Let $(M,\Omega,\Phi)$ be a K\"{a}hler Hamiltonian $K$-manifold with proper moment map.  
 We consider a regular element $a\in\tgot^*_{> 0}$ such as $a\notin \Delta(\Phi)$.
 Let $a'$ be the orthogonal projection of $a$ on $\Delta(\Phi)$. 
 
 We consider the K\"{a}hler Hamiltonian $K$-manifold $M\times \overline{Ka}$. Here $\overline{Ka}$ denotes the coadjoint orbit $Ka$ equipped with the opposite symplectic structure: 
 the complex structure is defined through the identifications $\overline{Ka}\simeq K/T\simeq K_\Cbb/B$, where $B$ is the Borel subgroup with Lie algebra 
 $\bgot=\tgot_\Cbb\oplus\ngot$. The moment map $\Phi_a: M\times \overline{Ka}\to\kgot^*$ is defined by the relation $\Phi_a(m,\xi)=\Phi(m)-\xi$.

We start with a basic result.
\begin{lem}
The function $\|\Phi_a\|:M\times \overline{Ka}\to\Rbb$ reaches its minimum on $K(\Phi^{-1}(a')\times\{a\})$. 
\end{lem}

So the minimal type of $M\times \overline{Ka}$ is 
$$
\gamma_a:=a'-a\neq 0.
$$ 
The corresponding critical set is $\Phi_a^{-1}(K\gamma_a)=K(\Phi^{-1}(a')\times\{a\})$. Let 
$C_{\gamma_a}$ be the connected component of $(M\times \overline{Ka})^{\gamma_a}$ containing $\Phi_a^{-1}(\gamma_a)$ : we have
 $C_{\gamma_a}=C_a\times \overline{K^{\gamma_a} a}$, where $C_a$ is 
the connected component of $M^{\gamma_a}$ containing $\Phi^{-1}(a')$.

The Bialynicki-Birula's complex submanifold $C_{\gamma_a}^-\subset M\times K_\Cbb/B$ is then 
$$
C_{\gamma_a}^-\simeq C_a^-\times P_{\gamma_a}/B\cap P_{\gamma_a}.
$$
We know from the previous section that $(\gamma_a,C_{\gamma_a})$ is a $K_\Cbb$-Ressayre's pair on $N\simeq M\times K_\Cbb/B$. 
Thanks to Proposition \ref{prop:K-B-RP}, we can translate this property on the manifold $M$. 

\begin{prop}\label{prop:ressayre-pair-gamma-a}
Let $a\in\tgot^*_{\geq 0}$ be a regular element such as $a\notin \Delta(\Phi)$. Then $(\gamma_a, C_a)$ is a $B$-Ressayre's pair of the complex  $K_\Cbb$-manifold $M$. If furthermore,
 $\gamma_a$ is rational and $\dim_K(C_a)-\dim_K(M)\in\{0,1\}$, then $(\gamma_a, C_a)$ is a regular $B$-Ressayre's pair
\end{prop}

\section{Proofs of the main theorems}

This section is devoted to the proofs of Theorems \ref{th:infinitesimal-ressayre-pairs} and \ref{th:ressayre-pairs}. 
Let $\Delta(\Phi):=\Phi(M)\cap\tgot^*_{\geq 0}$ be the Kirwan polyhedron of a K\"{a}hler Hamiltonian $K$-manifold $(M,\Omega)$ with proper moment map $\Phi$.

 We define the following subsets of the Weyl chamber. 

\begin{itemize}
\item $\Delta_{\infrp}$ is the set of points $\xi\in\tgot^*_{\geq 0}$ satisfying  the inequalities $\langle \xi,\gamma\rangle\geq \langle \Phi(C),\gamma\rangle$, for any 
{\bf infinitesimal  $B$-Ressayre's pair} $(\gamma,C)$.
\item $\Delta^{^{\reg}}_{\infrp}$ is the set of points $\xi\in\tgot^*_{\geq 0}$ satisfying  the inequalities $\langle \xi,\gamma\rangle\geq \langle \Phi(C),\gamma\rangle$, for any 
{\bf regular infinitesimal  $B$-Ressayre's pair} $(\gamma,C)$.
\item $\Delta_{\rp}$ is the set of points $\xi\in\tgot^*_{\geq 0}$ satisfying the inequalities $\langle \xi,\gamma\rangle\geq \langle \Phi(C),\gamma\rangle$, 
for any {\bf $B$-Ressayre's pair} $(\gamma,C)$.
\item $\Delta^{^{\reg}}_{\rp}$ is the set of points $\xi\in\tgot^*_{\geq 0}$ satisfying the inequalities $\langle \xi,\gamma\rangle\geq \langle \Phi(C),\gamma\rangle$, 
for any {\bf regular $B$-Ressayre's pair} $(\gamma,C)$.
\item $\Delta$ is the set of points $\xi\in\tgot^*_{\geq 0}$ satisfying the inequalities $\langle \xi,\gamma\rangle\geq \langle \Phi(C),\gamma\rangle$, 
for any {\bf regular $B$-Ressayre's pair} $(\gamma,C)$ satisfying $\Phi(C)\cap\tgot^*_{\geq 0}$.
\end{itemize}

By definition, we have the commutative diagram, where all the maps are inclusions:
$$
\xymatrix{
\Delta_{\infrp} \ar@{^{(}->}[d] \ar@{^{(}->}[r] & \Delta^{^{\reg}}_{\infrp} \ar@{^{(}->}[d] &  \\
\Delta_{\rp} \ar@{^{(}->}[r]          & \Delta^{^{\reg}}_{\rp}  \ar@{^{(}->}[r] & \Delta\cdot }
$$

In \S \ref{sec:step-1}, we prove the inclusions $ \Delta_{\rp}\subset\Delta(\Phi)\subset \Delta_{\infrp}$. It follows then that 
$$
\Delta(\Phi)=\Delta_{\infrp}=\Delta_{\rp}.
$$

In \S \ref{sec:step-2}, we prove the inclusion $\Delta\subset \Delta(\Phi)$, and since $\Delta(\Phi)\subset\Delta_{\rp}^{^{\reg}}\subset \Delta$, we get finally that 
$\Delta(\Phi)=\Delta^{^{\reg}}_{\rp}=\Delta$. At this stage, the proof of Theorem \ref{th:ressayre-pairs} is completed, together with the equivalence of the points {\em 1.} and 
{\em 2.} in Theorem \ref{th:infinitesimal-ressayre-pairs}.

In \S \ref{sec:RP-intersect}, we explain geometrically the equivalence of the points {\em 2.} and {\em 3.} in Theorem \ref{th:infinitesimal-ressayre-pairs}.

In \S \ref{sec:RP=infRP}, we expose in which circumstances an infinitesimal $B$-Ressayre's pair is a $B$-Ressayre's pair.

\subsection{$\Delta_{\rp}\subset\Delta(\Phi)\subset \Delta_{\infrp}$}\label{sec:step-1}

Let us check the two inclusions.
 
\medskip

\underline{$\Delta_{\rp}\subset\Delta(\Phi)$}

\medskip

Let $\xi\in\tgot^*_{\geq 0}$ that does not belong to $\Delta(\Phi)$. Let $r>0$ be the distance between $\xi$ and $\Delta(\Phi)$, also let $a\in\tgot^*_{> 0}$
be a regular element such as $\|\xi-a\|<\frac{r}{2}$. Since $a\notin \Delta(\Phi)$, we can exploit the result of \S \ref{sec:construction-RP}. 
Let $a'$ be the orthogonal projection of $a$ on $\Delta(\Phi)$. We consider $\gamma_a:=a'-a$, and the connected component $C_a$ of $M^{\gamma_a}$ 
containing $\Phi^{-1}(a')$. Proposition \ref{prop:ressayre-pair-gamma-a} tells us that  $(\gamma_a, C_a)$ is a $B$-Ressayre's pair.

Now, we compute 
\begin{eqnarray*}
\langle \xi,\gamma_a\rangle- \langle \Phi(C_a),\gamma_a\rangle&=&\langle \xi,\gamma_a\rangle- \langle a',\gamma_a\rangle\\
&=&\langle \xi-a,\gamma_a\rangle- \|\gamma_a\|^2\\
&\leq &- \|\gamma_a\|(\|\gamma_a\|- \|\xi-a\|)\\
&<&0\, .
\end{eqnarray*}
The last inequality comes from the fact that $\|\xi-a\|<\frac{r}{2}$ and that $\|\gamma_a\|>\frac{r}{2}$ since it represents the distance between $a$ and $\Delta(\Phi)$. 
The inequality $\langle \xi,\gamma_a\rangle- \langle \Phi(C_a),\gamma_a\rangle <0$ shows that $\xi\notin \Delta_{\rp}$.

\medskip

\underline{$\Delta(\Phi)\subset \Delta_{\infrp}$}

\medskip

Let $\xi\in \Delta(\Phi)$, and consider the K\"{a}hler manifold $N:=M\times \overline{K\xi}$. By definition the set $N^{ss}$ of analytical semi-stable points is dense in $N$.

Let $(\gamma, C)$ be an infinitesimal $B$-Ressayre's pair on $M$, and let $C_N:= C\times K_\Cbb^\gamma/P_\xi\cap K_\Cbb^\gamma$ be the corresponding connected component of 
$N^\gamma$. Here $P_\xi\subset K_\Cbb$ is the parabolic subgroup associated to $\xi\in\tgot^*_{\geq 0}$, so that $\overline{K\xi}\simeq K_\Cbb/P_\xi$. Let $C^-_N:=\{n\in N,  \lim_{t\to\infty} e^{-it\gamma} n\ \in C_N\}$ be the Bialynicki-Birula's submanifold.

\begin{lem}\label{lem:KCinterior}
\begin{enumerate}
\item The set $K_\Cbb C^-_N$ as a non-empty interior. 
\item $C^-_N\cap N^{ss}\neq \emptyset$.
\end{enumerate}
\end{lem}
\begin{proof}
The second point follows from the first one since $N^{ss}$ is a dense $K_\Cbb$-invariant subset of $N$.

Let $x\in C$ so that $\ngot^{\gamma>0}\cdot x\simeq (\T_x M)^{\gamma>0}$. The point $n=(x,[e])\in C_N$ is in the interior of $K_\Cbb C^-_N$ if we show that 
$\kgot_\Cbb\cdot n + \T_n C_N^-=\T_n N$. Since $\T_n C_N^-=(\T_n N)^{\gamma\leq 0}$ it is sufficient to check that 
$(\T_n N)^{\gamma>0}\subset \kgot_\Cbb\cdot n$. We have the decomposition $(\T_n N)^{\gamma>0}=(\T_x M)^{\gamma>0}\oplus \kgot_\Cbb^{\gamma>0}\cdot [e]$. It means that 
for any $v\in (\T_n N)^{\gamma>0}$, there exists $X\in\kgot_\Cbb^{\gamma>0}$ so that $v-X\cdot(x,[e])\in (\T_x M)^{\gamma>0}$. But  $\ngot^{\gamma>0}\cdot x\simeq (\T_x M)^{\gamma>0}$, so there exists $Y\in \ngot^{\gamma>0}$ such as $v-X\cdot(x,[e])=Y\cdot x$. The Lie algebra $\ngot$ is contained in the Lie algebra of the parabolic subgroup $P_\xi$: hence $Y\cdot [e]=0$. Finally we have proved that $v=(X+Y)\cdot(x,[e])\in  \kgot_\Cbb\cdot n$.
\end{proof}

Take $n\in C^-_N\cap N^{ss}$, and consider $n_\gamma=\lim_{t\to\infty}e^{-it\gamma} n\in C\times \{[e]\}$. The first point of Proposition \ref{prop:N-ss-gamma} tells us that 
$\langle \Phi_N(n_\gamma),\gamma\rangle\leq 0$, and this inequality means $\langle \Phi(C),\gamma\rangle\leq \langle\xi,\gamma\rangle$. 
We have proved that any point $\xi\in \Delta(\Phi)$ satisfies $\langle \Phi(C),\gamma\rangle\leq \langle\xi,\gamma\rangle$, for any infinitesimal 
$B$-Ressayre's pair $(\gamma,C)$. So, $\Delta(\Phi)\subset \Delta_{\infrp}$.

\subsection{$\Delta\subset \Delta(\Phi)$}\label{sec:step-2}

The aim of this section is the proof of the following 
\begin{theo}
Let $\xi\in\tgot^*_{\geq 0}$ satisfying the inequalities $\langle \xi,\gamma\rangle\geq \langle \Phi(C),\gamma\rangle$, for any regular $B$-Ressayre's pair $(\gamma,C)$ such as $\Phi(C)\cap \tgot^*_{\geq 0}\neq\emptyset$. Then $\xi\in \Delta(\Phi)$.
\end{theo}

Our arguments go as follows: we will show that there exists a collection $(\gamma_i, C_i)_{i\in I}$ of regular Ressayre's pairs satisfying 
$\Phi(C_i)\cap \tgot^*_{\geq 0}\neq\emptyset$, and  for which we have 
$$
\bigcap_{i\in I}\left\{\xi\in\tgot^*_{\geq 0}, \langle \xi,\gamma_i\rangle\geq \langle \Phi(C_i),\gamma_i\rangle\right\}=\Delta(\Phi).
$$
The set $I$ will be finite when $M$ is compact.

 We start with the following remark concerning admissible elements
 
\begin{lem}\label{lem:admissible}
$\dim_K(M^\gamma)=\dim_K(M)$ if and only if $K\cdot M^\gamma=M$.
\end{lem}

\begin{proof}
If $K\cdot M^\gamma=M$, we have obviously $\dim_K(M^\gamma)=\dim_K(M)$. Suppose now that there exists $x_o\in M^\gamma$ such as 
$\dim(\kgot_{x_o})= \dim_K(M)$. Then a neighborhood of $Kx_o$ is of the form $\Ucal=K\times_{K_{x_o}}V$ where $\kgot_{x_o}$ acts trivially on 
$V$ and $\gamma\in \kgot_{x_o}$. Then a neighborhood of $K^\gamma x_o$ in $M^\gamma$ is $\Vcal=K^\gamma\times_{K^\gamma_{x_o}}V$. We see then that 
$K\Vcal=\Ucal$: in other words $K\cdot M^\gamma$ contains $x_o$ in its interior. 

To any subalgebra $\hgot\subset \kgot$ we associate the sub-manifolds $M_\hgot=\{m\in M,\hgot=\kgot_m\}$ and $M^\hgot=\{m\in M,\hgot\subset
\kgot_m\}$. There exists a subalgebra $\hgot_o$, unique up to conjugation, such as $K\cdot M^{\hgot_o}=M$, and
$K\cdot M_{\hgot_o}$ is a dense open subset in $M$. The set $K\cdot M^\gamma\cap K\cdot M_{\hgot_o}$ is non-empty, hence, up to a change of $\hgot_o$ by conjugation, 
we have $M^\gamma\cap M_{\hgot_o}\neq\emptyset$: 
in other words $\gamma\in\hgot_o$. Finally we see that $M^{\hgot_o}\subset M^\gamma$ and $K\cdot M^\gamma=M$.
\end{proof}

\medskip

Before starting the description of the collection $(\gamma_i, C_i)_{i\in I}$, we need to recall the following facts. 
There exists a unique open face $\tau$ of the Weyl chamber $\tgot^*_{\geq 0}$ such as
$\Delta(\Phi)\cap\tau$ is dense in $\Delta(\Phi)$ : $\tau$ is called the {\em principal} face \cite{L-M-T-W}. All points in the open face $\tau$ have the same 
connected centralizer $K_\tau$. Let $A_\tau$ be the identity component of the center of $K_\tau$ and $[K_\tau,K_\tau]$
its semi-simple part. Note that we have an identification between the dual of the Lie algebra $\agot_\tau$ of $A_\tau$ and the linear span $\Rbb\tau$ of the face $\tau$. 
The Principal-cross-section Theorem \cite{L-M-T-W} tells us that $Y_\tau :=\Phi^{-1}(\tau)$ is a symplectic  $K_\tau$-manifold, with a trivial action of 
$[K_\tau,K_\tau]$. So, the restriction of $\Phi$ on $Y_\tau$ is a moment map $\Phi_\tau : Y_\tau\to \Rbb\tau$ for the Hamiltonian action of the torus $A_\tau$.

As the following example shows, the slice $Y_\tau$ is not always a complex submanifold of $M$.

\begin{exam} Let $\lambda$ be a regular element in the Weyl chamber, and consider the K\"{a}hler manifold $M=K(2\lambda)\times \overline{K\lambda}$. The moment map 
is defined by the relations $\Phi(\eta,\xi)=\eta-\xi$, and the principal face $\tau$ is the interior of the Weyl chamber. The element $x:=(2\lambda,\lambda)$ belongs to 
$Y_\tau:=\Phi^{-1}(\tau)$. Let $\qgot\subset \kgot$ be the image of $\ngot$ by the map $\Re:\kgot_\Cbb\to \kgot$. We view $\qgot$ as a complex $T$-module through the 
isomorphism $\Re:\ngot_\Cbb\simeq \qgot$. Then, the tangent space $\T_x M$ is equal to $\qgot\times \overline{\qgot}$ and the subspace 
$\T_x Y_\tau$ corresponds to $V:=\{(X,2X), X\in \qgot\}$. We see that $V$ is not a complex subspace of $\qgot\times \overline{\qgot}$.
\end{exam}

\subsubsection*{The Ressayre's pair $(\gamma_\tau, C_\tau)$}

In this section, we  show that a Ressayre's pair describes the fact that $\Delta(\Phi)$ is contained in $\bar{\tau}$. 
Let us denote by $H_\alpha\in \tgot$ the coroot associated to a root $\alpha$.  Let $\Rgot_\tau^+\subset \Rgot^+$ be the set of positive roots that are orthogonal to $\tau$.
\begin{defi}
Consider the following rational vector 
$$
\gamma_\tau:= -\sum_{\alpha\in\Rgot^+_\tau}H_\alpha.
$$
\end{defi}

\begin{lem}The element $\gamma_\tau$ satisfies the following properties:
\begin{itemize}
\item $\langle\xi,\gamma_\tau\rangle\leq 0$ for any $\xi\in\tgot^*_{\geq 0}$.
\item For any $\xi\in\tgot^*_{\geq 0}$, $\langle\xi,\gamma_\tau\rangle=0$ if and only if $\xi\in \bar{\tau}$.
\item $\gamma_\tau$  acts trivially on $Y_\tau$.
\item $\langle\alpha,\gamma_\tau\rangle< 0$ for any $\alpha\in\Rgot^+_\tau$.
\end{itemize}
\end{lem}
\begin{proof}
The first two points follow from the fact that $\langle\xi,H_\alpha\rangle\geq 0$ for any $\xi\in\tgot^*_{\geq 0}$ and any positive roots $\alpha$. The third point is due to the fact that $\gamma_\tau\in [\kgot_\tau,\kgot_\tau]$.

Let $\alpha_o$ be a simple root of $\Rgot^+_\tau$. Let $\sigma_{\alpha_o}:\tgot\to\tgot$ and  $\tilde{\sigma}_{\alpha_o}:\tgot^*\to\tgot^*$ be the associated orthogonal symmetries: 
$\sigma_{\alpha_o}(X)=X- \langle \alpha_o,X\rangle H_{\alpha_o}$, and $\tilde{\sigma}_{\alpha_o}(\xi)=\xi- \langle \xi,H_{\alpha_o}\rangle \alpha_o$. We have $\sigma_{\alpha_o}(H_\beta)=H_{ \tilde{\sigma}_{\alpha_o}(\beta)}$ for any root $\beta$. We also have  $\tilde{\sigma}_{\alpha_o}\left(\Rgot^+_\tau-\{\alpha_o\}\right)=\Rgot_\tau^+-\{\alpha_o\}$ and $\tilde{\sigma}_{\alpha_o}(\alpha_o)=-\alpha_o$. We  then see that 
$$
\sigma_{\alpha_o}\Big(\sum_{\alpha\in\Rgot^+_\tau} H_\alpha\Big)=\sum_{\alpha\in\Rgot^+_\tau}H_\alpha-2H_{\alpha_o}. 
$$
In other words $\langle\alpha_o,\sum_{\alpha\in\Rgot^+_\tau}H_\alpha\rangle=2$. Thus $\langle\alpha_o,\gamma_\tau\rangle<0$. This implies that 
$\langle\alpha,\gamma_\tau\rangle<0$ for any $\alpha\in \Rgot_\tau^+$.
\end{proof}

Let $C_\tau$ be the connected component of $M^{\gamma_\tau}$ containing $Y_\tau$. We start with the following basic result.

\begin{lem}\label{lem:gamma-tau}
\begin{itemize}
\item $\gamma_\tau$ is an admissible element.
\item For any $\xi\in\tgot^*_{\geq 0}$, the inequality $\langle \xi,\gamma_\tau\rangle\geq \langle \Phi(C_\tau),\gamma_\tau\rangle$
is equivalent to $\xi\in\bar{\tau}$.
\end{itemize}
\end{lem}

\begin{proof}
Since $K\cdot M^{\gamma_\tau}=M$, $\gamma_\tau$ is an admissible element (see Lemma \ref{lem:admissible}). 
Now we consider the inequality $\langle \xi,\gamma_\tau\rangle\geq \langle \Phi(C_\tau),\gamma_\tau\rangle$ for an element $\xi\in\tgot^*_{\geq 0}$. First we notice that
$\langle \Phi(C_\tau),\gamma_\tau\rangle=0$, and the first two points of the previous Lemma tell us that $\langle \xi,\gamma_\tau\rangle\geq 0$ is equivalent to 
$\xi\in\bar{\tau}$.
\end{proof}

\begin{prop}\label{prop:gamma-tau}
$(\gamma_\tau, C_\tau)$ is a regular $B$-Ressayre's pair such as $\Phi(C_\tau)\cap\tgot^*_{\geq 0}\neq \emptyset$.
\end{prop}

\begin{proof} Using the identification $\tgot\simeq \tgot^*$, we view $\gamma_\tau$ as a rational element of $\tgot^*$ orthogonal to $\tau$. 
Let $a'\in\Delta(\Phi)\cap\tau$ and consider the elements $a(n):=a'-\frac{1}{n}\gamma_\tau$ for $n\geq 1$. 
We notice that for $n$ large enough
\begin{enumerate}
\item $a(n)$ is a regular element of the Weyl chamber,
\item $a(n)\notin \Delta(\Phi)$,
\item $a'$ is the orthogonal projection of $a(n)$ on $\Delta(\Phi)$.
\end{enumerate}
So we can exploit the results of \S \ref{sec:construction-RP} with the elements $a(n)$ for $n>>1$. Proposition  
\ref{prop:ressayre-pair-gamma-a} and Lemma \ref{lem:gamma-tau} tell us that $(\gamma_\tau, C_\tau)$ is a regular $B$-Ressayre's pair.
\end{proof}

\subsubsection*{The Ressayre's pairs $(\gamma^\pm_l, C^\pm_l)$}

Let $S_\tau$ be the identity component of the principal stabilizer for the action of $A_\tau$ on $Y_\tau$: let $\sgot_\tau$ be its Lie algebra. 
The convex polyhedron $\Delta(\Phi)\subset \bar{\tau}$ generates an affine subspace $\Pi$ of $\Rbb\tau$ with direction $(\sgot_\tau)^\perp$.
In this section, we  show that a finite family of Ressayre's pairs describe the fact that $\Delta(\Phi)$ is contained in the affine subspace $\Pi$.

Let $(\eta_l)_{l\in L}$ be a rational basis of $\sgot_\tau$. We consider then the rational elements
$$
\gamma^\pm_l:= \pm \eta_l +\gamma_\tau,\quad l\in L.
$$
For any $l\in L$, we denote by $C^\pm_l$ the connected component of $M^{\gamma^\pm_l}$ containing $Y_\tau$.

\begin{lem}\label{lem:gamma-S}
\begin{itemize}
\item Any $\gamma^\pm_l$ is an admissible element.
\item The set of elements $\xi\in\Rbb\tau$ satisfying the inequalities
\begin{equation}\label{eq:gamma-S}
\langle \xi,\gamma_l^\pm\rangle\geq \langle \Phi(C^\pm_l),\gamma^\pm_j\rangle, \quad \forall l\in L
\end{equation}
corresponds to the affine subspace $\Pi$.
\end{itemize}
\end{lem}

\begin{proof} Since $K\cdot M^{S_\tau}=M$, we have $K\cdot M^{\gamma_l^\pm}=M$ for any $l\in L$. Hence, $(\gamma^\pm_l)_{l\in L}$ are  admissible elements.
Let $\xi_o\in \Phi(Y_\tau)\subset \Pi$. Since $\langle\xi,\gamma_\tau\rangle=0, \forall \xi\in\Rbb\tau$, the inequalities (\ref{eq:gamma-S}) are equivalent to
$\pm\langle \xi-\xi_o,\eta_l\rangle\geq 0,  \forall l\in L$: in other words $\xi-\xi_o\in (\sgot_\tau)^\perp$, so
$\xi\in\Pi$.
\end{proof}

\medskip

\begin{prop}\label{prop:gamma-S}
For any $l\in L$, $(\gamma^\pm_l, C^\pm_l)$ is a regular $B$-Ressayre's pair such as $\Phi(C^\pm_l)\cap\tgot^*_{\geq 0}\neq \emptyset$.
\end{prop}

\begin{proof}
The proof follows the lines of the proof of Proposition \ref{prop:gamma-tau}.  
Let $a'\in\Delta(\Phi)\cap\tau$ and consider the elements $a_l^\pm(n):=a'-\frac{1}{n}\gamma_l^\pm$ for $n\geq 1$. 
We notice that for $n$ is large enough
\begin{enumerate}
\item $a_l^\pm(n)$ is a regular element of the Weyl chamber,
\item $a_l^\pm(n)\notin \Delta(\Phi)$,
\item $a'$ is the orthogonal projection of $a_l^\pm(n)$ on $\Delta(\Phi)$.
\end{enumerate}
So we can exploit the results of \S \ref{sec:construction-RP} with the element $a_l^\pm(n)$ for $n>>1$. Proposition 
\ref{prop:ressayre-pair-gamma-a} tells us that $(\gamma^\pm_l, C^\pm_l)$ is a $B$-Ressayre's pair. 
\end{proof}

\subsubsection*{The Ressayre's pairs $(\gamma_F, C_F)$}

In this section, we show that the polyhedron $\Delta(\Phi)$, viewed as a subset of the affine subspace $\Pi$, is the intersection of 
the cone $\Pi\cap\tgot^*_{\geq 0}$, with a collection of half spaces parametrized by a family of $B$-Ressayre's pairs.

\begin{defi}
An open facet $F$ of $\Delta(\Phi)$ is called {\em non trivial} if $F\subset\tau$. We denote by $\Fcal(\Phi)$ the set of non trivial open facets of $\Delta(\Phi)$.
\end{defi}

Let $F\in \Fcal(\Phi)$. There exists $\eta_F\in\agot_\tau\setminus\sgot_\tau$ such as the affine space generated by $F$ is 
$\Pi_F=\{\xi\in \Pi,\langle\xi,\eta_F\rangle=\langle\xi_F,\eta_F\rangle\}$ for any $\xi_F\in F$. The vector $\eta_F$ is chosen so that 
$\Delta(\Phi)\subset \{\xi\in \Pi,\langle\xi,\eta_F\rangle\geq\langle\xi_F,\eta_F\rangle\}$.

By definition of the set $\Fcal(\Phi)$, we have the following description of the Kirwan polyhedron
\begin{equation}\label{eq:description-delta-K-M}
\Delta(\Phi)=  \bigcap_{F\in \Fcal(\Phi)}\left\{\xi\in \Pi,\langle\xi,\eta_F\rangle\geq\langle\xi_F,\eta_F\rangle\right\}\hspace{2mm}\bigcap\hspace{2mm}\tgot^*_{\geq 0}.
\end{equation}

The pull-back $Z_F:=\Phi_\tau^{-1}(F)$ is a connected component of the submanifold $Y_\tau^{\eta_F}$ fixed by $\eta_F$. Here, the generic infinitesimal 
stabilizer of the $A_\tau$-action on $Z_F$ is $\sgot_\tau\oplus\Rbb \eta_F$, hence the element $\eta_F$ can be chosen in $\wedge$. Consider an element $x_F\in Z_F$ 
with stabilizer subalgebra $\kgot_{x_F}$ equal  $\sgot_\tau\oplus\Rbb \eta_F\oplus [\kgot_\tau,\kgot_\tau]$.

We consider now the rational elements
$$
\gamma_F=\eta_F+\gamma_\tau,\quad F\in\Fcal(\Phi).
$$

Let $C_F$ be the connected component of $M^{\gamma_F}$ containing $Z_F$.

\begin{lem}\label{lem:gamma-F}
\begin{itemize}
\item $\gamma_F$ is an admissible element, for any $F\in \Fcal(\Phi)$.
\item The set of elements $\xi\in\Pi\cap \tgot^*_{\geq 0}$ satisfying the inequalities
\begin{equation}\label{eq:gamma-F}
\langle \xi,\gamma_F\rangle\geq \langle \Phi(C_F),\gamma_F\rangle, \quad \forall F\in \Fcal(\Phi)
\end{equation}
corresponds to $\Delta(\Phi)$.
\end{itemize}
\end{lem}

\begin{proof}
The element $x_F$ belongs to $C_F$, and the stabilizer subalgebra $\kgot_{x_F}$ is equal to $\sgot_\tau\oplus\Rbb \gamma_F\oplus [\kgot_\tau,\kgot_\tau]$. 
Hence $\dim_K(C_F)=\dim_K(M)+1$. The first point is settled. The last assertion of the proposition is a consequence of (\ref{eq:description-delta-K-M}).
\end{proof}

\begin{prop}\label{prop:gamma-F}
For any $F\in \Fcal(\Phi)$, the couple $(\gamma_F, C_F)$ is a regular $B$-Ressayre's pair such as $\Phi(C_F)\cap\tgot^*_{\geq 0}\neq \emptyset$.
\end{prop}

\begin{proof}
The proof follows the lines of the proof of Proposition \ref{prop:gamma-tau}.  
Let $a'_F\in F$ and consider the elements $a_F(n):=a'-\frac{1}{n}\gamma_F$ for $n\geq 1$. 
We notice that for $n$ is large enough
\begin{enumerate}
\item $a_F(n)$ is a regular element of the Weyl chamber,
\item $a_F(n)\notin \Delta(\Phi)$,
\item $a'_F$ is the orthogonal projection of $a_F(n)$ on $\Delta(\Phi)$.
\end{enumerate}
So we can exploit the results of \S \ref{sec:construction-RP} with the elements $a_F(n)$ for $n>>1$. Proposition 
\ref{prop:ressayre-pair-gamma-a} tells us that $(\gamma_F, C_F)$ is a $B$-Ressayre's pair. 
\end{proof}

\subsection{Ressayre's pairs and semi-stable points II}\label{sec:RP-semi-stable-2}

Let $N$ be a K\"{a}hler Hamiltonian $K$-manifold  with proper moment map $\Phi_N$. 
The aim of this section is the following proposition that completes the results of Proposition \ref{prop:N-ss-gamma}.

\begin{prop}\label{prop:N-ss-RP}
Let $C_N$ be a connected component of $N^\gamma$. Then, the following holds
\begin{enumerate}
\item If $N^{ss}\cap C_N^- \neq\emptyset$, then $\langle \Phi_N(C_N),\gamma\rangle\leq 0$.
\item If $N^{ss}\cap C_N^-\neq\emptyset$ and $\langle \Phi_N(C_N),\gamma\rangle= 0$, then $C_N\cap\Phi_N^{-1}(0)\neq\emptyset$.
\end{enumerate}
\end{prop}

\begin{proof}
Let $x\in C_N^-\cap N^{ss}$, and let denote by $x_\gamma\in C_N$ the limit of $e^{-it\gamma} x$ when $t\to\infty$.
In Proposition \ref{prop:N-ss-gamma}, we proved that $\langle \Phi_N(x_\gamma),\gamma\rangle\leq 0$, and also that 
$x_\gamma\in N^{ss}$ when $\langle \Phi_N(x_\gamma),\gamma\rangle=0$. Our first point is then settled. 

Suppose that $x_\gamma\in N^{ss}$: it remains to show that $C_N\cap\Phi_N^{-1}(0)\neq\emptyset$. Let $z\in \Phi_N^{-1}(0)$ so that 
$K_\Cbb z$ is the closed orbit contained in $\overline{K_\Cbb x_\gamma}\subset \overline{K_\Cbb x}$. 
We will use now the Holomorphic Slice Theorem of Sjamaar \cite{Sjamaar-95} in order to describe a neighbourhood of $K_\Cbb z$. 
First recall that the stabilizer subgroup $(K_\Cbb)_{z}$ is equal to the complexification of $H:= K_{z}$.
Let $E\subset \T_{z}N$ be the orthogonal  complement (relatively to the K\"{a}hler metric) of the subspace $\kgot_\Cbb\cdot z$. We see that 
$E$ is a Hermitian vector space equipped with a linear action of $H_\Cbb$. Sjamaar proved in \cite{Sjamaar-95} that a $K_\Cbb$-invariant 
neighbourhood $\Wcal$ of the orbit $K_\Cbb\, z$ admits a $K_\Cbb$-equivariant holomorphic diffeomorphism with
$$
\widetilde{\Wcal}:=K_\Cbb\times_{H_\Cbb}E^0,
$$
where $E^0$ is a $H_\Cbb$-invariant neighborhood of $0\in E$.

By definition, the orbits $K_\Cbb\, x$  and $K_\Cbb\, x_\gamma$ intersect $\Wcal$, so the points $x$ and $x_\gamma$ belong to $\Wcal$. 
Let us write by $\tilde{x},\tilde{x}_\gamma,\tilde{z}\in \widetilde{\Wcal}$ the image of $x,x_\gamma,z\in \Wcal$. Here $\tilde{z}=[e,0]$.

\begin{lem}\label{lem:holomorphic-model}
\begin{enumerate}
\item There exists $(a,p_\gamma,k_\gamma, v)\in K\times P_\gamma\times K^\gamma_\Cbb\times E$ such as
$$
\tilde{x}=[p_\gamma a,v]\quad {\rm and}\quad \tilde{x}_\gamma=[k_\gamma a,v_o].
$$ 
The data $(a,p_\gamma,k_\gamma, v)$ satisfies the following relations: $\gamma':=a^{-1}\gamma\in\hgot$,  $v\in E^{\gamma'\leq 0}$, 
$v_o:=\lim_{t\to\infty}e^{-it\gamma'} v\in E^{\gamma'= 0}$, and $k_\gamma:=\lim_{t\to\infty}e^{-it\gamma} p_\gamma e^{it\gamma}$.
\item $0\in \overline{H_\Cbb\, v}$ and $0\in \overline{H^{\gamma'}_\Cbb\, v_o}$.
\item $K\,\tilde{z}\cap \overline{K^\gamma_\Cbb\, \tilde{x}_\gamma}\neq\emptyset$.
\end{enumerate}
\end{lem}

The third point of the lemma says that $K z\cap \overline{K^\gamma_\Cbb x_\gamma}\neq\emptyset$. But $\overline{K^\gamma_\Cbb x_\gamma}\subset C_N$ and 
$K z\subset \Phi_N^{-1}(0)$, so we have proved that $C_N\cap\Phi_N^{-1}(0)\neq\emptyset$. The proof of Proposition \ref{prop:N-ss-RP} is completed.

\begin{proof} Consider the projection $T:\widetilde{\Wcal}\to K_\Cbb/H_\Cbb$. Since $\lim_{t\to\infty}e^{-it\gamma} \tilde{x}=\tilde{x}_\gamma$, we have 
$\lim_{t\to\infty}e^{-it\gamma} T(\tilde{x})= T(\tilde{x}_\gamma)\in (K_\Cbb/H_\Cbb)^\gamma$. It is not difficult to check that there exists 
$(a,p_\gamma,k_\gamma)\in K\times P_\gamma\times K^\gamma_\Cbb$ so that
$$
T(\tilde{x})=[p_\gamma a]\quad {\rm and}\quad  T(\tilde{x}_\gamma)=[k_\gamma a],
$$
where $k_\gamma:=\lim_{t\to\infty}e^{-it\gamma} p_\gamma e^{it\gamma}$ and $\gamma':=a^{-1}\gamma\in\hgot$.

Now let $v,v_o\in E$ such as $\tilde{x}=[p_\gamma a,v]$ and $\tilde{x}_\gamma=[k_\gamma a,v_o]$. Since 
$\lim_{t\to\infty}e^{-it\gamma} \tilde{x}=\tilde{x}_\gamma$, we have $\lim_{t\to\infty}e^{-it\gamma'} v= v_o$, and this implies that 
$v\in E^{\gamma'\leq 0}$ and $v_o\in E^{\gamma'= 0}$. The first point of the lemma is settled.

Since $\tilde{z}$ belongs to $\overline{K_\Cbb\tilde{x}}$, we must have $0\in \overline{H_\Cbb\, v}$. Let $P_{\gamma'}$ be the parabolic subgroup of 
$K_\Cbb$ attached to $\gamma'$, and take $P_{\gamma'}^H:=P_{\gamma'}\cap H_\Cbb$. Since $H_\Cbb=H \, P_{\gamma'}^H$, we have 
$0\in \overline{P_{\gamma'}^H v}$. Let $\epsilon>0$, and let $w= p\, v\in P_{\gamma'}^H v$ such as $\|w\|\leq \epsilon$. Since $w$ belongs to 
$E^{\gamma'\leq 0}$, we have $\|e^{-it\gamma'} w\|\leq \epsilon$ for any $t\geq 0$. Hence the vector $w_o:= \lim_{t\to\infty}e^{-it\gamma'} w$ satisfies 
$\|w_o\|\leq \epsilon$ and $w_o= h v_o$, where $h=\lim_{t\to\infty}e^{-it\gamma'} p \,e^{it\gamma'}\in H_\Cbb^{\gamma'}$. We have proved that 
$0\in \overline{H^{\gamma'}_\Cbb\, v_o}$, and this fact implies that $a \tilde{z}=[a,0]$ belongs to $\overline{K^\gamma_\Cbb\, \tilde{x}_\gamma}$. The last point is proved. 
\end{proof}
\end{proof}

\begin{rem}\label{rem:choisir-z}
In the proof of Lemma \ref{lem:holomorphic-model}, we have seen that the element $z\in\overline{K_\Cbb\, x}\cap\Phi^{-1}(0)$ can be chosen in $\overline{K_\Cbb^\gamma\, x_\gamma}\cap\Phi^{-1}(0)$. 
If we do so, the expression in $\widetilde{\Wcal}$ are $\tilde{x}=[p,v]$ and $\tilde{x}_\gamma=[k,v_o]$, where $p\in P_\gamma$ and 
$k:=\lim_{t\to\infty}e^{-it\gamma} p \,e^{it\gamma}\in K_\Cbb^\gamma$.
\end{rem}

\medskip

\subsection{{\em 2.} $\Leftrightarrow$ {\em 3.} in Theorem \ref{th:infinitesimal-ressayre-pairs}.}\label{sec:RP-intersect}

Let $M$ be a K\"{a}hler Hamiltonian $K$-manifold with proper moment map $\Phi$. We denote by $\Rcal$ the set of infinitesimal $B$-Ressayre's pair on $M$, and by 
$\Rcal'\subset\Rcal$ the subset formed by the infinitesimal $B$-Ressayre's pair $(\gamma,C)$ satisfying $\Phi(C)\cap\tgot^*_{\geq 0}$.

In \S \ref{sec:step-1} and  \S \ref{sec:step-2}, we proved that $\Delta_{\infrp}=\Delta(\Phi)$. The aim of this section is to explain why $\Delta_{\infrp}$ coincides with 
$$
\Delta'_{\infrp} = \left\{\xi\in\tgot^*_{\geq 0};\ \langle \xi,\gamma\rangle\geq \langle \Phi(C),\gamma\rangle,\ \forall (\gamma,C)\in\Rcal'\right\}.
$$

To any $(\gamma, C)\in \Rcal$, we associate the hyperplane of $\tgot^*$ : $\Pi_{(\gamma, C)}:=$ \break 
$\left\{\xi\in\tgot^*;\ \langle \xi,\gamma\rangle= \langle \Phi(C),\gamma\rangle\right\}$.
Let $\Rcal^0\subset \Rcal$ be the subset formed by the infinitesimal $B$-Ressayre's pair $(\gamma,C)$ satisfying $\Pi_{(\gamma, C)}\cap \Delta(\Phi)\neq \emptyset$.

The equality $\Delta_{\infrp}=\Delta'_{\infrp}$ follows from the following proposition.
\begin{prop}
\begin{enumerate}
\item $\Delta_{\infrp}$ is equal to 
$$
\Delta^0_{\infrp} = \left\{\xi\in\tgot^*_{\geq 0};\ \langle \xi,\gamma\rangle\geq \langle \Phi(C),\gamma\rangle,\ \forall (\gamma,C)\in\Rcal^0\right\}.
$$
\item $\Rcal'=\Rcal^0$.
\end{enumerate}
\end{prop}

\begin{proof}
We have obviously $\Delta_{\infrp}\subset \Delta^0_{\infrp}$. Let us consider $\xi_0\in \Delta^0_{\infrp}$ and $\xi_1\in \Delta_{\infrp}$. For any $t\in [0,1]$, let's 
$\xi_t=t\xi_1+(1-t)\xi_0$. Let $s=\inf\{t\in [0,1], \xi_t\in \Delta(\Phi)\}$. If $s> 0$, there exists $(\gamma, C)\in \Rcal$ such as 
$\langle \xi_s,\gamma\rangle =\langle \Phi(C),\gamma\rangle$ and $\langle \xi_{s-\epsilon},\gamma\rangle <\langle \Phi(C),\gamma\rangle$ for $\epsilon>0$ small enough. The first equality means that $\xi_s\in \Pi_{(\gamma, C)}\cap \Delta(\Phi)$, hence $(\gamma, C)\in\Rcal^0$. Then the second inequality is in contradiction with the fact that 
$\langle \xi_0,\gamma\rangle \geq \langle \Phi(C),\gamma\rangle$ and $\langle \xi_1,\gamma\rangle \geq \langle \Phi(C),\gamma\rangle$. So, $s=0$ which means that 
$\xi_0\in \Delta_{\infrp}$. The first point is settled.

By definition, we have $\Rcal'\subset \Rcal^0$. Let us prove the opposite inclusion. 
Let $(\gamma,C)\in\Rcal^0$. The hyperplane $\Pi_{(\gamma, C)}$ intersects $\Delta(\Phi)$ and we want to prove that $\Phi(C)\cap\tgot_{\geq 0}^*\neq 0$. Let us 
consider $\xi\in \Pi_{(\gamma, C)}\cap\Delta(\Phi)$ and the associated manifold $N:=M\times\overline{K\xi}$. We consider the connected component 
$C_N=C\times\overline{K^\gamma\xi}$ of $N^\gamma$, and the Bialynicki-Birula's submanifold
$C^-_N$. We prove in Lemma \ref{lem:KCinterior} that $C^-_N\cap N^{ss}\neq \emptyset$ and  the relation $\langle \xi,\gamma\rangle = \langle \Phi(C),\gamma\rangle$ means that 
$\langle\Phi_N(C_N),\gamma\rangle=0$. We can conclude, thanks to Proposition \ref{prop:N-ss-RP}, that $0\in\Phi_N(C_N)$: in other words $\xi\in \Phi(C)$.
\end{proof}

\subsection{Ressayre's pairs $=$ infinitesimal Ressayre's pairs}\label{sec:RP=infRP}

Let $M$ be a K\"{a}hler Hamiltonian $K$-manifold with proper moment map $\Phi$. The main purpose of this section is the following proposition. 

\begin{prop}\label{prop:RP=infRP}
Let $(\gamma,C)$ be an infinitesimal $B$-Ressayre's pair on $M$ satisfying both conditions:
\begin{enumerate}
\item $K_\Cbb C^-$ is dense in $M$,
\item $\Phi(C)\cap\tgot^*_{>0}\neq\emptyset$.
\end{enumerate}
Then $(\gamma,C)$ be a $B$-Ressayre's pair on $M$.
\end{prop}

\begin{rem}
Ressayre obtained a similar result in the algebraic setting (see \cite{Ressayre10}, \S 6).
\end{rem}

\begin{proof} Let $\xi\in \Phi(C)\cap\tgot^*_{>0}$, and consider the K\"{a}hler Hamiltonian $K$-manifold $N:= M\times \overline{K\xi}$. We work with the connected component 
$C_N:=C\times \overline{K^\gamma\xi}$. Thanks to Proposition \ref{prop:K-B-RP}, we know that it is sufficient to check that $(\gamma,C_N)$ is a $K_\Cbb$-Ressayre's pair on $N$.
We will use here, as a key fact, that the stabilizer subgroups for the $K$-action on $N$ are abelian: in particular, for any $n\in C_N$, the stabilizer subgroup $K_n$ is an abelian subgroup contained in $K_\gamma$.

By definition the set $N^{ss}$ is non-empty and $\langle\Phi_N(C_N),\gamma\rangle=0$. Let $p:C^-_N\to C_N$ be the canonical projection.
We have proved in Proposition \ref{prop:N-ss-RP} that $N^{ss}\cap C_N^-=p^{-1}(C^{ss}_N)$, where $C^{ss}_N=\{n\in C_N; \overline{K^\gamma_\Cbb\, n}\cap \Phi_N^{-1}(0)\}$ 
is a dense open subset of $C_N$.

\begin{lem}
Let $g\in K_\Cbb$ and $x\in N^{ss}\cap C_N^-$, such as $gx\in C_N^-$. Then $g\in P_\gamma$.
\end{lem}
\begin{proof}Let $x'=gx$. Let $x_\gamma:=\lim_{t\to\infty}e^{-it\gamma}x$ and $x'_\gamma:=\lim_{t\to\infty}e^{-it\gamma}x'$ be the corresponding point in $C_N^{ss}$. 
The orbits $K_\Cbb x_\gamma$ and $K_\Cbb x'_\gamma$ are both contained in the closure of the orbit $K_\Cbb x\subset N^{ss}$. 
Thanks to Proposition \ref{prop:N-ss-RP} and Remark \ref{rem:choisir-z}, we know that there exists $z\in C_N\cap \Phi_N^{-1}(0)$, such as 
$z\in \overline{K^\gamma_\Cbb x_\gamma}\cap \overline{K^\gamma_\Cbb x'_\gamma}$. 

As in the proof of Proposition \ref{prop:N-ss-RP}, we use now the Holomorphic Slice Theorem of Sjamaar \cite{Sjamaar-95} in order to describe a neighbourhood of $K_\Cbb z$. 
The stabilizer subgroup $(K_\Cbb)_{z}$, which is equal to the complexification of $H:= K_{z}$, is an abelian group contained in $K^\gamma_\Cbb$.
A $K_\Cbb$-invariant neighbourhood $\Wcal$ of the orbit $K_\Cbb z$ admits a $K_\Cbb$-equivariant holomorphic diffeomorphism with
$\widetilde{\Wcal}:=K_\Cbb\times_{H_\Cbb}E^0$.

The points $x,x'$ belong to $\Wcal$. Let us write by $\tilde{x},\tilde{x}'$ the corresponding points in $\widetilde{\Wcal}$. Lemma \ref{lem:holomorphic-model} 
and Remark \ref{rem:choisir-z} say that there exists $(p,p')\in P_\gamma^2$  such as $\tilde{x}=[p,v]$, and $\tilde{x}'=[p',v']$. But $x'=gx$, so there exists 
$h\in H_\Cbb$ satisfying $gp=p'h$. Finally, we see that $g=p'hp^{-1}\in P_\gamma$ because $h\in K_\Cbb^\gamma$. 
\end{proof}

Let us finish the proof of Proposition \ref{prop:RP=infRP}.  The map $\pi_\gamma: K_\Cbb\times_{P_\lambda}C_N^-\to N$ has a dense image by hypothesis, and it is injective when restricted to $K_\Cbb\times_{P_\lambda}(N^{ss}\cap C_N^-)$. Let $(C^-_N)_{\reg}$ be the open subset formed by the point $n\in C^-_N$ such as the tangent map 
$\T{\rm \pi}_\gamma\vert_{[e, n]}$ is an isomorphism. We have seen in Lemma \ref{lem:C-N-reg} that 
$(C^-_N)_{\reg}$ is a dense, $P_\gamma$-invariant, open subset of $C^-_N$ such as $(C^-_N)_{\reg}=p^{-1}(C_N\cap (C^-_N)_{\reg})$.

Finally, let's take $U=(C^-_N)_{\reg}\cap N^{ss}$. We have $U=p^{-1}(V)$ where $V=C^{ss}_N\cap(C^-_N)_{\reg}$ is a dense open subset of $C_N$. The map
$\pi_\gamma$ defines then an holomorphic diffeomorphism between $K_\Cbb\times_{P_\lambda}U$ and the dense open subset $K_\Cbb U\subset N$.
\end{proof}

\section{H\"{o}rn conditions}\label{sec:Horn-conditions}

Let $\gamma\in\tgot$ be a non-zero element, and let $C\subset M^\gamma$ be a connected component. 
Let $C^-$ be the Bialynicki-Birula's complex manifold associated to $C$ (see (\ref{eq:BB})). In order to analyse the fibers of the map 
$q_\gamma : B\times_{B\cap P_\gamma}C^-\to M$, we introduce the maps
 $r^g_\gamma : B\times_{B\cap P_\gamma}C^-\to K_\Cbb/Ad(g)P_\gamma$ that 
sends $[b,x]$ to $[bg^{-1}]$.

\subsection{Characterization of $B$-Ressayre's pairs}\label{sec:Characterization-RP}
We start with the following basic remark.

\begin{lem}\label{lem:r-gamma}
Let $g\in K_\Cbb$. For any $m\in M$, the map $r^g_\gamma$ is injective when restricted to $q_\gamma^{-1}(m)$.
\end{lem}

We have the following characterization of $B$-Ressayre's pairs.

\begin{prop}\label{prop:caracteriser-ressayre-pair}Let $g\in K_\Cbb$.
\begin{itemize}
\item If $(\gamma,C)$ is an infinitesimal $B$-Ressayre's pair then :
\begin{enumerate}
\item[A)] $\dim_\Cbb(\ngot^{\gamma>0})={\rm rank}_\Cbb(\T M\vert_C)^{\gamma>0}$,
\item[B)] $\tr_{\gamma}(\ngot^{\gamma>0})=\tr_{\gamma}((\T M\vert_C)^{\gamma>0})$.
\end{enumerate}
\item If $(\gamma,C)$ satisfies conditions A), B), and furthermore the set 
\begin{equation}\label{eq:condition-1-image}
{\rm Image}(q_\gamma)=\{m\in M, r^g_\gamma(q_\gamma^{-1}(m))\neq\emptyset\}\quad \  has\ a\ nonempty\ interior,
\end{equation}
then $(\gamma,C)$ is an infinitesimal $B$-Ressayre's pair.
\item Suppose that $M$ is an algebraic variety equipped with an algebraic action of $K_\Cbb$. If $(\gamma,C)$ satisfies relations A), B), and furthermore 
the set 
\begin{equation}\label{eq:condition-2-image}
\{m\in M, {\rm cardinal}(r^g_\gamma(q_\gamma^{-1}(m)))= 1\}\quad \ has\ a\ nonempty\ interior,
\end{equation}
then $(\gamma,C)$ is a $B$-Ressayre's pair.
\end{itemize}
\end{prop}

\begin{proof}  Let $x\in C$ such as  $\rho_x^\gamma:=\ngot^{\gamma>0}\to (\T_x M)^{\gamma>0}$ is bijective. 
Since the map $\rho_x^\gamma$ commutes with the action of $\gamma$ we get A) and B). The first point is proved.

Suppose now that $(\gamma,C)$ satisfies conditions A) and B).  Condition A) insures that the manifolds $B\times_{B\cap P_\gamma}C^-$ and $M$ have the same dimension.
Either (\ref{eq:condition-1-image}) or (\ref{eq:condition-2-image}) imply that ${\rm Image}(q_\gamma)$ has a non-empty interior, so it must contains a regular value of the map $q_\gamma$. We see then that the set $(C^-)_{\reg}$ formed by the point $m\in C^-$ for which the tangent map $\T q_\gamma\vert_{[e, m]}$ is an isomorphism, is a dense open subset of $C^-$.

As we did in Lemma \ref{lem:C-N-reg}, let us check that  equality B) implies that $(C^-)_{\reg}\cap C\neq \emptyset$. 
The rank $s$ of the holomorphic bundle $\mathbb{E}:=\T M\vert_{C^-}/\T C^-$ is equal to the dimension of 
$\bgot/\pgot_\gamma\cap\bgot$. We consider now the $B\cap P_\gamma$-holomorphic line bundle $\Lbb\to C^-$ defined by 
$\Lbb:=\hom \left(\wedge^r(\bgot/\pgot_\gamma\cap\bgot), \wedge^r\mathbb{E}\right)$. We have a canonical $B\cap P_\gamma$-equivariant section 
$\theta : C^-\to \Lbb$ defined  by
$\theta(m) : \overline{X_1} \wedge \cdots\wedge  \overline{X_s}\longrightarrow  \overline{X_1\cdot m} \wedge \cdots\wedge  \overline{X_s\cdot m}$. 
We notice that  $(C^-)_{\reg}=\{m; \theta(m)\neq 0\}$ and condition $B)$ tell us that the torus $\Tbb=\overline{\exp(\Rbb\gamma)}$ acts trivially on $\Lbb\vert_{C}$. 
Thanks to Lemma \ref{lem:action-fibre-L}, we can conclude that $m\in (C^-)_{\reg}$ if and only if $p(m)\in (C^-)_{\reg}$. The set $C\cap (C^-)_{\reg}$ is not empty, and for any 
$x\in C\cap (C^-)_{\reg}$, the map $\rho^\gamma_x$ is an isomorphism. At this stage, we have proved that $(\gamma,C)$ is an infinitesimal $B$-Ressayre's pair.

Suppose now that $M$ is an algebraic variety. We have an algebraic map $q_\gamma$ between smooth algebraic varieties of same dimension. Moreover, on the open subset 
$B\times_{B\cap P_\gamma}(C^-)_{\reg}$, the differential $q_\gamma$ is always bijective. If (\ref{eq:condition-2-image}) holds, we can conclude that $q_\gamma$ defines an isomorphism between $B\times_{B\cap P_\gamma}(C^-)_{\reg}$ and $B(C^-)_{\reg}$. We have proved that $(\gamma,C)$ is a $B$-Ressayre's pair since $(C^-)_{\reg}\cap C\neq \emptyset$.
\end{proof}

\medskip

\subsection{Highest weight vector}

\begin{defi}
We associate to an infinitesimal Ressayre's pair $(\gamma,C)$ the holomorphic line bundle $\Lbb_{(\gamma,C)}:=\det\left((\T M\vert_C)^{\gamma>0}\right)$ on $C$.
\end{defi}

Take a basis $Y_1,\cdots,Y_p$ of $\ngot^{\gamma>0}$. We define an holomorphic section of the line bundle $\Lbb_{(\gamma,C)}$ as follows :
$\Theta_{(\gamma,C)}(x):= Y_1\cdot x\wedge\cdots\wedge  Y_p\cdot x,\quad \forall x\in C$.

The vector space of holomorphic sections $H^0(C, \Lbb_{(\gamma,C)})$ is a $K_\gamma$-module, and we will now see that $\Theta_{(\gamma,C)}\in H^0(C, \Lbb_{(\gamma,C)})$ is a highest weight vector.
Let 
$$
\rho_{(\gamma,C)}:=\sum_{\stackrel{\alpha>0}{\langle\alpha,\gamma\rangle<0}}\alpha.
$$

The following property is immediate.

\begin{prop}
We have $X\cdot \Theta_{(\gamma,C)}=0$ for any $X\in\ngot^{\gamma=0}$ and 
$$
t\cdot  \Theta_{(\gamma,C)}= t^{\rho_{(\gamma,C)}}\ \Theta_{(\gamma,C)},\quad \forall t\in T.
$$
\end{prop}

\section{Examples}

\subsection{$\tilde{K}_\Cbb\times K_\Cbb$ acting on $\tilde{K}_\Cbb$}\label{sec:example-1}

In this section, we work out the example of a complex reductive group $\tilde{K}_\Cbb$ equipped with the following action of 
$\tilde{K}_\Cbb\times K_\Cbb$ : $(\tilde{k},k)\cdot a= \tilde{k} a k^{-1}$. Here $K\croc \tilde{K}$ is a closed connected subgroup, and 
$K_\Cbb$ is the corresponding reductive subgroup of $\tilde{K}_\Cbb$. 

There is a diffeomorphism of the cotangent bundle $\T^*\tilde{K}$ with $\tilde{K}_\Cbb$ defined as follows. We identify $\T^*\tilde{K}$ with 
$\tilde{K}\times \tilde{\kgot}^*$ by means of left-translation and then with $\tilde{K}\times \tilde{\kgot}$ by means of an invariant inner product on  $\tilde{\kgot}$. 
The map $\varphi:\tilde{K}\times \tilde{\kgot}\to \tilde{K}_\Cbb$ given by $\varphi(a,X)=a e^{iX}$ is a diffeomorphism. If we use $\varphi$ to transport the complex structure of 
$\tilde{K}_\Cbb$ to $\T^*\tilde{K}$, then the resulting complex structure on $\T^*\tilde{K}$ is compatible with the symplectic structure on $\T^*\tilde{K}$, so that $\T^*\tilde{K}$ 
becomes a K\"{a}hler Hamiltonian $\tilde{K}\times K$-manifold (see \cite{Hall87}, \S 3). The moment map relative to the $\tilde{K}\times K$-action is a proper 
map $\Phi=\Phi_{\tilde{K}} \oplus \Phi_K : \T^*\tilde{K} \to \tilde{\kgot}^*\oplus \kgot^*$ defined by
\begin{equation}\label{eq:momentcotangent}
\Phi_{\tilde{K}}(a,\tilde{\xi})=-a\tilde{\xi},\quad \Phi_K(a,\tilde{\xi})=\pi(\tilde{\xi}).
\end{equation}
Here $\pi : \tilde{\kgot}^* \to \kgot^*$ is the projection dual to the inclusion $\kgot \croc \tilde{\kgot}$ of Lie algebras.

Select maximal tori $T$ in $K$ and $\tilde{T}$ in $\tilde{K}$ such as $T\subset \tilde{T}$, and
Weyl chambers $\tilde{\tgot}^*_{\geq 0}$ in $\tilde{\tgot}^*$ and
$\tgot^*_{\geq 0}$ in  $\tgot^*$, where $\tgot$ and $\tilde{\tgot}$ denote the Lie algebras of $T$, resp. $\tilde{T}$.

Let $\Delta(\T^*\tilde{K})\subset \tilde{\tgot}^*_{\geq 0}\times \tgot^*_{\geq 0}$
be the Kirwan polytope associated to $\Phi$. Equations (\ref{eq:momentcotangent}) show that
$$
\Delta(\T^*\tilde{K})=\left\{(\tilde{\xi},\xi)\in\tilde{\tgot}^*_{\geq 0}\times \tgot^*_{\geq 0}\, |\, -\xi\in \pi\big(\tilde{K}\tilde{\xi}\big) \right\}.
$$

We recall from general principle that the Kirwan polyhedron $\Delta(\T^*\tilde{K})$ has a non zero interior in $\tilde{\tgot}^*\times \tgot^*$ if and only if
the generic stabilizer of the $\tilde{K}\times K$-action on $\T^* \tilde{K}$ is finite. To simplify the exposition, we assume from now on that 
{\em no non-zero ideal of $\kgot$ is an ideal of $\tilde{\kgot}$}. It implies the following
\begin{lem}
The intersection of  $\Delta(\T^*\tilde{K})$ with $\tilde{\tgot}^*_{>0}\times \tgot^*_{>0}$ has a non empty interior, and the generic stabilizer of the 
$\tilde{K}\times K$-action on $\tilde{K}_\Cbb\simeq \T^* \tilde{K}$ is finite.
\end{lem}

In the next section, we look after the $B$-Ressayre's pairs of the complex manifold $\tilde{K}_\Cbb$ relative to the $\tilde{K}_\Cbb\times K_\Cbb$-action.

\subsubsection{Admissible elements}

Let $(\tilde X,X)$ be an element of  $\tilde \tgot\times \tgot$.
In the following lemma, which follows by direct checking, we describe the manifold $N^{(\tilde X,X)}$ of zeroes of the vector field $(\tilde X,X)$ on $N=\T^*\tilde K$. 
We denote $\tilde{K}^X\subset \tilde{K}$ the subgroup that fixes $X$. The cotangent bundle $\T^*\tilde{K}^X\simeq \tilde{K}^X_\Cbb$ is then a submanifold of 
$\T^*\tilde{K}\simeq \tilde{K}_\Cbb$.

\begin{lem}\label{lem:point-fixe-cotangent}

\begin{itemize}
\item If $\tilde X$ is not conjugate to $X$ in $\tilde \kgot$ , then
$N^{(\tilde X,X)}=\emptyset$

\item If $\tilde X=aX$ with $a\in \tilde K$, then
$N^{(\tilde X,X)}=a\cdot \T^*\tilde{K}^X$.
\end{itemize}
\end{lem}

Let $\tilde{W}=N(\tilde{T})/\tilde{T}$ be the Weyl group. We notice that if $\tilde{X}=aX$ for some $a\in \tilde K$, then there exists $\tilde{w}\in \tilde{W}$ such as 
$\tilde{X}=\tilde{w}X$. So, the admissible elements relative to the action of $\tilde{K}_\Cbb\times K_\Cbb$ on $\tilde{K}_\Cbb$ are of the form $(\tilde{w}\gamma,\gamma)$ where 
$\gamma\in \tgot$ is a rational element and  $\tilde{w}\in \tilde{W}$.

We consider now a $K$-invariant decomposition at the level of Lie algebras: 
$$
\tilde{\kgot}=\kgot\oplus\qgot
$$
Let $\qgot^\gamma\subset\qgot$ be the subspace fixed by the adjoint action of $\gamma$. We let $\tilde{\Rgot}\subset \tilde{\tgot}^*$ denote the set of roots for the group $\tilde{K}_\Cbb$, and $\tilde{\Rgot}^{\gamma=0}$ denote the subset of roots vanishing on $\gamma$.

\begin{lem}\label{admissible-exemple-1}Let $\gamma\in \tgot$ be a rational element and  let $\tilde{w}\in \tilde{W}$. 
\begin{enumerate}
\item The element $(\tilde{w}\gamma,\gamma)$ is admissible if and only if $\dim_T(\qgot^\gamma)=1$.
\item If $(\tilde{w}\gamma,\gamma)$ is admissible, then the projection of $\tilde{\Rgot}^{\gamma=0}$ on $\tgot^*$ generates $(\Rbb\gamma)^\perp$.
\end{enumerate}
\end{lem}

\begin{coro}
The $\tilde{K}\times K$-action on $\tilde{K}_\Cbb$ admits a finite number of admissible elements (modulo identifications $(\tilde{w}q\gamma,q\gamma)\sim (\tilde{w}\gamma,\gamma)$ for any $q\in\Qbb^{>0}$).
\end{coro}

\begin{proof}
Since the generic stabilizer of the $\tilde{K}\times K$-action on $\tilde{K}_\Cbb$ is finite, we have $\dim_{\tilde{K}\times K}(\tilde{K}_\Cbb)=0$. The submanifold fixed by  
$(\tilde{w}\gamma,\gamma)$ is  $\tilde{w}\tilde{K}^\gamma_\Cbb$ and a direct computation gives 
\begin{eqnarray*}
\dim_{\tilde{K}\times K}(\tilde{w}\tilde{K}^\gamma_\Cbb)=\dim_{\tilde{K}\times K}(\tilde{K}^\gamma_\Cbb)=\dim_{\tilde{K}\times K}(\tilde{K}^\gamma\times \tilde{\kgot}^\gamma)
&=& \dim_{\tilde{K}\times K}(\{1\}\times \tilde{\kgot}^\gamma)\\
&=& \dim_{K}(\tilde{\kgot}^\gamma)\\
&=& \dim_{K}(\kgot^\gamma\times \qgot^\gamma)=\dim_T(\qgot^\gamma).
\end{eqnarray*}
As $\dim_{\tilde{K}\times K}(\tilde{K}_\Cbb)=0$, a rational element $(\tilde{w}\gamma,\gamma)$ is admissible if and only if 
$\dim_T(\qgot^\gamma)=$ $\dim_{\tilde{K}\times K}(\tilde{w}\tilde{K}^\gamma_\Cbb)$ is equal to $1$ (see Remark \ref{rem:cas-generic-stabilizer}). 
The first point is proved.

The projection of $\tilde{\Rgot}^{\gamma=0}$ on $\tgot^*$ is equal to $\Rgot^{\gamma=0}\cup\Rgot(\qgot^{\gamma})$ where $\Rgot^{\gamma=0}$ denotes the set of roots
 for the group $K$ vanishing on $\gamma$ and 
$\Rgot(\qgot^{\gamma})$ denotes the set of weights for the $T$-action  on $\qgot^{\gamma}\otimes\Cbb$. The equality $\dim_T(\qgot^\gamma)=1$ means that the set 
$\Rgot(\qgot^{\gamma})$ generates $(\Rbb\gamma)^\perp$. So, the first point tells us that the projection of $\tilde{\Rgot}^{\gamma=0}$ on $\tgot^*$ generates 
$(\Rbb\gamma)^\perp$, when $(\tilde{w}\gamma,\gamma)$ is admissible.
\end{proof}

\subsubsection{Ressayre's pairs of $\tilde{K}_\Cbb$}

At any element $(\tilde{w},\gamma)\in \tilde{W}\times \tgot$, we associate $\gamma_{\tilde{w}}:=(\tilde{w}\gamma,\gamma)$ and 
 the submanifold $C_{\tilde{w},\gamma}:=\tilde{w}\tilde{K}^\gamma_\Cbb\subset \tilde{K}_\Cbb$ fixed by $\gamma_{\tilde{w}}$. 
 The Bialynicki-Birula's complex submanifold associated to $C_{\tilde{w},\gamma}$ is $C_{\tilde{w},\gamma}^-:= \tilde{w}\tilde{P}_\gamma$ 
where $\tilde{P}_\gamma$ is the parabolic subgroup of $\tilde{K}_\Cbb$ associated to $\gamma$ (see (\ref{eq:P-gamma})). The parabolic subgroup of 
$\tilde{K}_\Cbb\times K_\Cbb$ associated to the weight $\gamma_{\tilde{w}}$ is 
$$
\Pbb_{\tilde{w},\gamma}:=\mathrm{Ad}(\tilde{w})(\tilde{P}_\gamma)\times P_\gamma.
$$

Let $B\subset K_\Cbb$ (resp. $\tilde{B}\subset \tilde{K}_\Cbb$) be the Borel subgroup associated to the choice of the Weyl chamber $\tgot^*_{\geq 0}$
 (resp. $\tilde{\tgot}^*_{\geq 0}$).
 We can now consider the holomorphic map
 $$
{\rm q}_{\tilde{w},\gamma}: (\tilde{B}\times B)\times_{(\tilde{B}\times B)\cap \Pbb_{\tilde{w},\gamma}} C_{\tilde{w},\gamma}^-\longrightarrow  \tilde{K}_\Cbb
$$
that sends $[\tilde{b},b;\tilde{w}\tilde{p}]$ to $\tilde{b}\tilde{w}\tilde{p}b^{-1}$. 

 In our context, Relations A) and B) of Proposition \ref{prop:caracteriser-ressayre-pair} can be computed as follows.

\begin{lem}\label{lem:A-B-horn}
\begin{itemize}
\item Relation A) is  $\dim_\Cbb(\tilde{\ngot}^{\tilde{w}\gamma>0})+\dim_\Cbb (\ngot^{\gamma>0})=$ 
$\dim_\Cbb (\tilde{\kgot}_\Cbb)^{\tilde{w}\gamma>0}$.
\item Relation B) corresponds to 
\begin{equation}\label{eq:trace-condition-w-tilde}
\sum_{\stackrel{\alpha\in\Rgot^+}{\langle\alpha,\gamma\rangle> 0}}\langle\alpha,\gamma\rangle=
\sum_{\stackrel{\tilde{\alpha}\in\tilde{\Rgot}^-}{\langle\tilde{\alpha},\tilde{w}\gamma\rangle> 0}}\langle\tilde{\alpha},\tilde{w}\gamma\rangle.
\end{equation}
\end{itemize}
\end{lem}

In the next section, we will see that the $B$-Ressayre's pair on $\tilde{K}_\Cbb$ have a nice caracterisation in terms of Schubert calculus.

\subsubsection{Schubert calculus}

For any $\gamma\in\tgot$, we consider the projective varieties $\Fcal_\gamma:=K_\Cbb/P_\gamma$, $\tilde{\Fcal}_\gamma:=\tilde{K}_\Cbb/\tilde{P}_\gamma$, 
and the map
$$
{\rm r}_{\tilde{w},\gamma}: (\tilde{B}\times B)\times_{(\tilde{B}\times B)\cap \Pbb_{\tilde{w},\gamma}} C_{\tilde{w},\gamma}^-\longrightarrow  \tilde{\Fcal}_\gamma \times \Fcal_\gamma
$$
that sends $[\tilde{b},b;\tilde{w}\tilde{p}]$ to $([\tilde{b}\tilde{w}],[b])$. Here ${\rm r}_{\tilde{w},\gamma}$ corresponds to the map $r^g_\gamma$ introduced in \S 
\ref{sec:Horn-conditions}, with $g=(\tilde{w}^{-1},e)$.

We associate to any $\tilde{w}\in\tilde{W}$, the Schubert cell
$$
\tilde{\Xgot}^o_{\tilde{w},\gamma}:= \tilde{B}[\tilde{w}]\subset \tilde{\Fcal}_\gamma.
$$
and the Schubert variety $\tilde{\Xgot}_{\tilde{w},\gamma}:=\overline{\tilde{\Xgot}^o_{\tilde{w},\gamma}}$. If $\tilde{W}^\gamma$ denotes the subgroup of $\tilde{W}$ that fixes 
$\gamma$, we see that  the Schubert cell $\tilde{\Xgot}^o_{\tilde{w},\gamma}$ and the Schubert variety $\tilde{\Xgot}_{\tilde{w},\gamma}$ depends only of the class of $\tilde{w}$ in 
$\tilde{W}/\tilde{W}^\gamma$.

On the variety $\Fcal_\gamma$, we consider the Schubert cell $\Xgot^o_{\gamma}:= B[e]$ and the Schubert variety $\Xgot_{\gamma}:=\overline{\Xgot^o_{\gamma}}$. 
Notice that we have a canonical embedding $\iota: \Fcal_\gamma\to  \tilde{\Fcal}_\gamma$ since $K_\Cbb\cap \tilde{P}_\gamma=P_\gamma$.

\begin{lem}\label{lem:schubert-1}
\begin{itemize}
\item $(\gamma_{\tilde{w}}, C_{\tilde{w},\gamma})$ satisfies Relations A) if and only if $\dim_\Cbb(\tilde{\Fcal}_\gamma)=$ \break $\dim_\Cbb(\tilde{\Xgot}^o_{\tilde{w},\gamma})+\dim_\Cbb(\Xgot^o_{\gamma})$.

\item For any $\tilde{k}\in \tilde{K}_\Cbb$, we have a bijective correspondance
${\rm q}_{\tilde{w},\gamma}^{-1}(\tilde{k})\simeq \tilde{k}^{-1}\tilde{\Xgot}^o_{\tilde{w},\gamma}\cap \iota(\Xgot^o_{\gamma})$.
\end{itemize}
\end{lem}
\begin{proof} 
The first point is consequence of the first point of Lemma \ref{lem:A-B-horn}. Thanks to Lemma \ref{lem:r-gamma}, we know that ${\rm r}_{\tilde{w},\gamma}$ is injective 
when restricted to ${\rm q}_{\tilde{w},\gamma}^{-1}(\tilde{k})$. Now, it is an easy matter to check that ${\rm r}_{\tilde{w},\gamma}\big({\rm q}_{\tilde{w},\gamma}^{-1}(\tilde{k})\big)$ 
corresponds to the set $\{(x,y)\in \tilde{\Xgot}^o_{\tilde{w},\gamma}\times \Xgot^o_{\gamma}, x=\tilde{k}y\}$. The second point is completed.
\end{proof}

We consider the cohomology\footnote{Here, we use singular cohomology with integer coefficients.} ring $H^*(\tilde{\Fcal}_\gamma,\Zbb)$ of $\tilde{\Fcal}_\gamma$. 
If Y is an irreducible closed subvariety of $\tilde{\Fcal}_\gamma$, we denote by $[Y]\in H^{2n_Y}(\tilde{\Fcal}_\gamma,\Zbb)$ its cycle class in cohomology : here $n_Y={\rm codim}_\Cbb(Y)$. Let $\iota^*:H^*(\tilde{\Fcal}_\gamma,\Zbb)\to H^*(\Fcal_\gamma,\Zbb)$ be the pull-back map in cohomology. 
Recall that the cohomology class $[pt]$ (resp. $[\tilde{pt}]$) associated to a singleton $Y=\{pt\}\subset \Fcal_\gamma$  (resp.  $\tilde{Y}=\{\tilde{pt}\}$) is a basis of 
$H^{\maxx}(\Fcal_\gamma,\Zbb)$ (resp. $H^{\maxx}(\tilde{\Fcal}_\gamma,\Zbb)$).

\medskip

We recall some classical properties.

\begin{lem}\label{lem:schubert-2}
\begin{enumerate}
\item Let $\tilde{X},\tilde{Y}$ be two irreducible closed subvarieties of $\tilde{\Fcal}_\gamma$. Let's $\tilde{X}^o,\tilde{Y}^o$ be their smooth part.
\begin{itemize}
\item We have $[\tilde{X}]\cdot [\tilde{Y}]=[\tilde{X}\cap g \tilde{Y}]$ for $g\in \tilde{K}_\Cbb$ belonging to a dense open subset. 
\item The relation $[\tilde{X}]\cdot [\tilde{Y}]=n[\tilde{pt}]$ holds if and only if the set 
$\tilde{X}^o\cap g \tilde{Y}^o$ is of cardinal $n$, for $g\in \tilde{K}_\Cbb$ belonging to a dense open subset. 
\end{itemize}
\item Let $Y$ (resp. $\tilde{Y}$) be an irreducible closed subvariety of $\Fcal_\gamma$ (resp. $\tilde{\Fcal}_\gamma$). The relation $[\iota(Y)]\cdot [\tilde{Y}]= 
n[\tilde{pt}]$ in $H^*(\tilde{\Fcal}_\gamma,\Zbb)$ is equivalent to the relation $[Y]\cdot \iota^*([\tilde{Y}])= n[pt]$ in $H^*(\Fcal_\gamma,\Zbb)$.
\end{enumerate}
\end{lem}

\begin{proof} The first point is a consequence of Kleiman transversality theorem \cite{Kleiman} (see \cite{BK1}, Proposition 3). The second point is left to the reader.
\end{proof}

\medskip

Lemmas \ref{lem:schubert-1}, \ref{lem:schubert-2} and Proposition \ref{prop:caracteriser-ressayre-pair} give us the following corollary.

\begin{coro}\label{coro:critere-schubert-ressayre}
\begin{itemize}
\item $(\gamma_{\tilde{w}}, C_{\tilde{w},\gamma})$ is a $B$-Ressayre's pair  on $\tilde{K}_\Cbb$ if and only if (\ref{eq:trace-condition-w-tilde}) holds and $[\Xgot_{\gamma}]\cdot \iota^*([\tilde{\Xgot}_{\tilde{w},\gamma}])= [pt]$.
\item $(\gamma_{\tilde{w}}, C_{\tilde{w},\gamma})$ is an infinitesimal $B$-Ressayre's pair on $\tilde{K}_\Cbb$ if and only if (\ref{eq:trace-condition-w-tilde}) holds and $[\Xgot_{\gamma}]\cdot \iota^*([\tilde{\Xgot}_{\tilde{w},\gamma}])= n[pt]$ for $n\geq 1$.
\end{itemize}
\end{coro}

\medskip 

We can finally describe the Kirwan polyhedron $\Delta(\T^*\tilde{K})$.

\medskip

\begin{theo}\label{theo:delta-exemple-1} Let $K\subset \tilde{K}$ be a closed connected subgroup such as {\em no non-zero ideal of $\kgot$ is an ideal of $\tilde{\kgot}$}. Let $(\tilde{\xi},\xi)\in\tilde{\tgot}^*_{\geq 0}\times \tgot^*_{\geq 0}$. We have  $-\xi\in \pi\big(\tilde{K}\tilde{\xi}\big)$ if and only if 
$$
\langle \tilde{\xi},\tilde{w}\tilde{\gamma}\rangle+\langle \xi,\gamma\rangle\geq 0
$$
for any $(\gamma,\tilde{w})\in\tgot\times \tilde{W}$ satisfying the following properties:
\begin{enumerate}
\item[a)] $\gamma$ is rational and $\dim_T(\qgot^\gamma)=1$. 
\item[b)] $[\Xgot_{\gamma}]\cdot \iota^*([\tilde{\Xgot}_{\tilde{w},\gamma}])= [pt]$ in $H^*(\Fcal_\gamma,\Zbb)$.
\item[c)] $\sum_{\stackrel{\alpha\in\Rgot^+}{\langle\alpha,\gamma\rangle> 0}}\langle\alpha,\gamma\rangle=
\sum_{\stackrel{\tilde{\alpha}\in\tilde{\Rgot}^-}{\langle\tilde{\alpha},\tilde{w}\gamma\rangle> 0}}\langle\tilde{\alpha},\tilde{w}\gamma\rangle$.
\end{enumerate}
\end{theo}

The previous theorem has a long story. The first input was given by Klyachko \cite{Klyachko} with a refinement by Belkale \cite{Belkale01}, when $K_\Cbb=SL(n)$ and 
$\tilde{K}_\Cbb=(SL(n))^s$. The case $\tilde{K}_\Cbb=(K_\Cbb)^s$ was treated by Kapovich-Leeb-Millson \cite{Ka-Le-Mi} following an analogous slightly weaker result 
proved by Berenstein-Sjamaar \cite{Berenstein-Sjamaar}. Condition c) is related to the notion of Levi-movability introduced by Belkale-Kumar \cite{BK1}. It is recalled 
that the notion of Ressayre's pair is an adaptation of Belkale-Kumar's Levi-movability. Finally, the general case was treated by Ressayre \cite{Ressayre10}, 
where he proves furthermore the irredundancy of the list of inequalities.

\medskip

\subsection{$\tilde{K}_\Cbb\times K_\Cbb$ acting on $\tilde{K}_\Cbb\times V$}

Let $(V,h)$ be an Hermitian $K$-vector space. The linear action of $K_\Cbb$ on $V$ is denoted 
$\rho : K_\Cbb\to {\rm GL}(V)$.

Let $\Omega_V=-{\rm Im}(h)$ be the corresponding $2$-form on $V$. The moment map 
$\Phi_V:V\to\kgot^*$ associated to the action of $K$ on $(V,\Omega_V)$ is defined by the relation
$$
\langle\Phi_V(v),X\rangle=\frac{1}{2}\Omega(Xv,v),\quad X\in\kgot.
$$

We suppose here that the moment map $\Phi_V$ is {\em proper}. It is like saying that the algebra ${\rm Sym}(V^*)^K$ of invariant polynomial functions on $V$ is reduced to the constants.

In this section, we study the action of $\tilde{K}_\Cbb\times K_\Cbb$ on $\tilde{K}_\Cbb\times V\simeq \T^* \tilde{K}\times V$ that is defined as follows : 
$(\tilde{k},k)\cdot (a,v)= (\tilde{k} a k^{-1},kv)$. The corresponding moment map $\Phi: \T^* \tilde{K}\times V\to \tilde{\kgot}^*\oplus \kgot^*$, which is defined by the relations
\begin{equation}\label{eq:moment-map-exemple-2}
\Phi(\tilde{k},\tilde{\xi},v)=\left(-\tilde{k}\tilde{\xi},\pi(\tilde{\xi})+\Phi_V(v)\right),
\end{equation}
is proper.

\subsubsection{Admissible elements}

In this example, Lemmas \ref{admissible-exemple-1} and \ref{lem:A-B-horn} become

\begin{lem}\label{admissible-horn-exemple-2}Let $\gamma\in \tgot$ be a rational element and let $\tilde{w}\in \tilde{W}$. 
\begin{enumerate}
\item The element $\gamma_{\tilde{w}}:=(\tilde{w}\gamma,\gamma)$ is admissible if and only if $\dim_T(\qgot^\gamma\times V^\gamma)=1$.
\item Relation A) means $\dim_\Cbb(\tilde{\ngot}^{\tilde{w}\gamma>0})+\dim_\Cbb (\ngot^{\gamma>0})= 
\dim_\Cbb (\tilde{\kgot}_\Cbb)^{\tilde{w}\gamma>0}+\dim_\Cbb(V^{\gamma>0})$.
\item Relation B) corresponds to 
\begin{equation}\label{eq:trace-condition-w-tilde-exemple-2}
\sum_{\stackrel{\alpha\in\Rgot^+}{\langle\alpha,\gamma\rangle> 0}}\langle\alpha,\gamma\rangle=
\sum_{\stackrel{\tilde{\alpha}\in\tilde{\Rgot}^-}{\langle\tilde{\alpha},\tilde{w}\gamma\rangle> 0}}\langle\tilde{\alpha},\tilde{w}\gamma\rangle+\tr_\gamma(V^{\gamma> 0}).
\end{equation}
\end{enumerate}
\end{lem}

\subsubsection{Ressayre's pairs of $\tilde{K}_\Cbb\times V$}

The submanifold fixed by $\gamma_{\tilde{w}}$ is $\tilde{w}\tilde{K}^\gamma_\Cbb\times V^\gamma$, and 
the corresponding Bialynicki-Birula's complex submanifold is $\tilde{w}\tilde{P}_\gamma\times V^{\gamma\leq 0}$. So, we work with the holomorphic map
 $$
{\rm q}^V_{\tilde{w},\gamma}: (\tilde{B}\times B)\times_{(\tilde{B}\times B)\cap \Pbb_{\tilde{w},\gamma}} (\tilde{w}\tilde{P}_\gamma\times V^{\gamma\leq 0})
\longrightarrow  \tilde{K}_\Cbb\times V
$$
that sends $[\tilde{b},b;\tilde{w}\tilde{p},v]$ to $(\tilde{b}\tilde{w}\tilde{p}b^{-1},\rho(b)v)$. 

In order to analyse the fibers of the map ${\rm q}^V_{\tilde{w},\gamma}$, we use,  as in \S  \ref{sec:example-1}, the map 
${\rm r}^V_{\tilde{w},\gamma}: (\tilde{B}\times B)\times_{(\tilde{B}\times B)\cap \Pbb_{\tilde{w},\gamma}} (\tilde{w}\tilde{P}_\gamma\times V^{\gamma\leq 0})\longrightarrow  
\tilde{\Fcal}_\gamma \times \Fcal_\gamma$ that sends $[\tilde{b},b;\tilde{w}\tilde{p}, v]$ to $([\tilde{b}\tilde{w}],[b])$.

We start with

\begin{lem}
The map ${\rm r}^V_{\tilde{w},\gamma}$ is injective when it is restricted to the fibers of ${\rm q}^V_{\tilde{w},\gamma}$.
\end{lem}

In order to obtain a criteria based on Schubert calculus, we replace $V$ by the projective manifold $\Pbb(\Cbb\times V)=G/Q$. Here 
$G={\rm GL}(\Cbb\times V)$, and $Q\subset G$ is the parabolic subgroup of $G$ that fixes $[1,0]\in \Pbb(\Cbb\times V)$. Let $P_\gamma^G \subset G$ be the 
parabolic subgroup associated to $\gamma$. Then $V^\gamma$ (resp. $V^{\gamma\leq 0}$) is a dense open part of $\Pbb(\Cbb\times V^\gamma)=G^\gamma\cdot [1,0]$ 
(resp.  $\Pbb(\Cbb\times V^{\gamma\leq 0})=P^G_\gamma\cdot [1,0]$).

Let us work with the holomorphic map
 $$
\widehat{\rm q}^V_{\tilde{w},\gamma}: 
(\tilde{B}\times B)\times_{(\tilde{B}\times B)\cap \Pbb_{\tilde{w},\gamma}} (\tilde{w}\tilde{P}_\gamma\times P^G_{\gamma}/Q\cap P^G_{\gamma})
\longrightarrow  \tilde{K}_\Cbb\times G/Q
$$
that sends $[\tilde{b},b;\tilde{w}\tilde{p},[p]]$ to $(\tilde{b}\tilde{w}\tilde{p}b^{-1},[\widehat{\rho}(b)p])$. Here we denote $\widehat{\rho}:K_\Cbb\to {\rm GL}(\Cbb\times V)$ 
the composition of $\rho : K_\Cbb\to {\rm GL}(V)$, with the morphism ${\rm GL}(V)\to{\rm GL}(\Cbb\times V)$ that sends $g$ to 
$
\begin{pmatrix}
1 & 0\\
0 & g
\end{pmatrix}.
$

By an argument of density, we can replace ${\rm q}^V_{\tilde{w},\gamma}$ by $\widehat{\rm q}^V_{\tilde{w},\gamma}$, and obtain the following adaptation of Proposition 
\ref{prop:caracteriser-ressayre-pair}. Let $N:=\tilde{K}_\Cbb\times G/Q$.
\begin{prop}\label{prop:caracteriser-ressayre-pair-exemple-2} Suppose that $\gamma_{\tilde{w}}=(\tilde{w}\gamma,\gamma)$ satisfies Conditions A) and B).
\begin{itemize}
\item If the set ${\rm Image}(\widehat{\rm q}^V_{\tilde{w},\gamma})=\{n\in N, {\rm r}^V_{\tilde{w},\gamma}((\widehat{\rm q}^V_{\tilde{w},\gamma})^{-1}(n))\neq\emptyset\}$  has a 
 nonempty interior, then $(\gamma_{\tilde{w}}, \tilde{w}\tilde{K}^\gamma_\Cbb\times V^\gamma)$ is an infinitesimal $B$-Ressayre's pair.
\item  If the set $\{n\in N, {\rm cardinal}({\rm r}^V_{\tilde{w},\gamma}((\widehat{\rm q}^V_{\tilde{w},\gamma})^{-1}(n))= 1\}$ has a 
 nonempty interior, then then $(\gamma_{\tilde{w}}, \tilde{w}\tilde{K}^\gamma_\Cbb\times V^\gamma)$ is a $B$-Ressayre's pair.
\end{itemize}
\end{prop}

\subsubsection{Schubert calculus}

To an oriented real vector bundle $\Ecal\to N$ of rank $r$, we can associate its Euler class ${\rm Eul}(\Ecal)\in H^{2r}(N,\Zbb)$. 
When $\Vcal\to N$ is a complex vector bundle, then ${\rm Eul}(\Vcal_\Rbb)$ corresponds to the top Chern class $c_{p}(\Vcal)$, where $p$ is the complex rank of $\Vcal$, and 
$\Vcal_\Rbb$ means $\Vcal$ as a real vector bundle oriented by its complex structure (see \cite{Bott-Tu}, \S 21).

We will need the following result.

\begin{lem}
Let $M$ be a smooth complex projective variety. Let $Y\subset M$ be a subvariety and let $Y^o\subset M$ be it's smooth part. Let 
$\iota_X : X\to M$ be a smooth compact oriented submanifold contained in $Y^o$. Then 
$$
\iota_X^*([Y])={\rm Eul}(\Ncal\vert_X),
$$
where $\Ncal$ is the normal bundle of $Y^o$ in $M$.
\end{lem}
\begin{proof}
We will give a proof of this equality in the de Rham cohomology of $M$. If $\eta$ is a closed form on $M$, the cycle class $[Y]$ is defined through the relation
$\int_M [Y] \cdot \eta = \int_{Y^o} \eta\vert_{Y^o}$.The convergence of the right hand side is a consequence Lelong integrability's theorem of algebraic cycles. 
Let $\Ncal_{(X,M)}$ be the normal bundle of $X$ in $M$. 
The de Rham cohomology with compact support $H^*_{c}(\Ncal_{(X,M)})$ is freely generated, as a $H^*_{c}(X)$-module, by the Thom class ${\rm Thom}(X,M)$. We denote by 
$Thom(X,M)\in  H^*_c(\Ucal_X)$, the image of ${\rm Thom}(X,M)$ through a diffeomorphism between $\Ncal_{(X,M)}$ and a tubular neighborhood $\Ucal_X$ of $X$ in $M$.
If $\theta\in H^*_{c}(X)$, we denote simply by $Thom(X,M)\cdot \theta\in  H^*_c(\Ucal_X)$ the image of ${\rm Thom}(X,M)\cdot p^*(\theta)\in H^*_{c}(\Ncal_{(X,M)})$: here $p:\Ncal_{(X,M)}\to X$ denotes the projection.
Now, we compute $\int_M[Y]\cdot Thom(X,M)\cdot \theta$ in two different ways. Since $Thom(X,M)$ is the de Rham class that represents the oriented submanifold 
$X$, we have $\int_M[Y]\cdot Thom(X,M)\cdot \theta=\int_X\iota_X^*([Y])\cdot\theta$. On the other hand, we have 
$\int_M[Y]\cdot Thom(X,M)\cdot \theta= \int_{Y^o} (Thom(X,M)\cdot \theta)\vert_{Y^o}$. As $Thom(X,M)=Thom(X,Y^o)\cdot Thom(Y^o,M)$, we obtain  
$Thom(X,M)\vert_{Y^o}=Thom(X,Y^o)\cdot {\rm Eul}(\Ncal)$ because $Thom(Y^o,M)\vert_{Y^o}={\rm Eul}(\Ncal)$. Finally, 
$$
\int_X\iota_X^*([Y])\cdot\theta=\int_M[Y]\cdot Thom(X,M)\cdot \theta=\int_X{\rm Eul}(\Ncal)\vert_X\cdot\theta,\quad \forall \theta\in H^*(X).
$$
This relation completes the proof. 
\end{proof}

The isomorphism $V^{\gamma>0}\simeq V/V^{\gamma\leq 0}$ shows that $V^{\gamma>0}$ can be viewed as a $P_\gamma$-module. Let 
$\Vcal^{\gamma>0}= K_\Cbb\times_{P_\gamma}V^{\gamma>0}$ be the corresponding complex vector bundle on $\Fcal_\gamma$.
In the following proposition, we denote simply ${\rm Eul}(V^{\gamma>0})$ the Euler class ${\rm Eul}(\Vcal^{\gamma>0}_\Rbb)\in H^*(\Fcal_\gamma,\Zbb)$

\begin{prop}
\begin{itemize}
\item $(\gamma_{\tilde{w}},  \tilde{w}\tilde{K}^\gamma_\Cbb\times V^\gamma)$ is a $B$-Ressayre's pair  on $\tilde{K}_\Cbb\times V$ if and only if 
(\ref{eq:trace-condition-w-tilde-exemple-2}) holds and $[\Xgot_{\gamma}]\cdot \iota^*([\tilde{\Xgot}_{\tilde{w},\gamma}])\cdot {\rm Eul}(V^{\gamma>0})= [pt]$ in $H^*(\Fcal_\gamma,\Zbb)$.
\item $(\gamma_{\tilde{w}}, \tilde{w}\tilde{K}^\gamma_\Cbb\times V^\gamma)$ is an infinitesimal $B$-Ressayre's pair on $\tilde{K}_\Cbb\times V$ if and only if 
(\ref{eq:trace-condition-w-tilde-exemple-2}) holds and $[\Xgot_{\gamma}]\cdot \iota^*([\tilde{\Xgot}_{\tilde{w},\gamma}])\cdot {\rm Eul}(V^{\gamma>0})= n [pt]$ for $n\geq 1$.
\end{itemize}
\end{prop}
\begin{proof}
As we noticed in Proposition \ref{prop:caracteriser-ressayre-pair-exemple-2}, we have to analyse the sets $\theta(\tilde{k},g):=
{\rm r}^V_{\tilde{w},\gamma}((\widehat{\rm q}^V_{\tilde{w},\gamma})^{-1}(\tilde{k},g[e]))\subset \tilde{\Fcal}_\gamma \times \Fcal_\gamma$ associated to 
$(\tilde{k},g)\in \tilde{K}_\Cbb\times G$. 

Consider the maps $\iota :  \Fcal_\gamma\to \tilde{\Fcal}_\gamma$ and $j:\Fcal_\gamma\to \Fcal_\gamma^G$ : here $\Fcal_\gamma^G:=G/P^G_\gamma$ and 
$j([k])=[\hat{\rho}(k)]$. Let $Y^o\subset \Fcal_\gamma^G$ be the orbit $Q\cdot[e]$ and let $Y:=\overline{Y^o}$. A direct computation shows that we have a bijective correspondance 
$$
\theta(\tilde{k},g)\simeq   (\iota\times j)(\Xgot^o_{\gamma})\bigcap \tilde{k}^{-1}\tilde{\Xgot}^o_{\tilde{w},\gamma}\times gY^o.
$$
As in Lemma \ref{coro:critere-schubert-ressayre}, we see that $(\gamma_{\tilde{w}},  \tilde{w}\tilde{K}^\gamma_\Cbb\times V^\gamma)$ is a $B$-Ressayre's pair if and only if 
(\ref{eq:trace-condition-w-tilde-exemple-2}) holds and 
$$
[(\iota\times j)(\Xgot_{\gamma})]\cdot([\tilde{\Xgot}_{\tilde{w},\gamma}]\times [Y])= [pt] \times [pt]\in H^*(\tilde{\Fcal}_\gamma,\Zbb)\times H^*(\Fcal^G_\gamma,\Zbb)
$$
The former relation is equivalent to $[\Xgot_{\gamma}]\cdot \iota^*([\tilde{\Xgot}_{\tilde{w},\gamma}])\cdot j^*([Y])= [pt] \in H^*(\Fcal_\gamma,\Zbb)$.
In order to compute the form $j^*([Y])$, we decompose the map $j:\Fcal_\gamma\to \Fcal_\gamma^G$ in $j=\iota_X\circ j_X$. Here $X= {\rm GL}(V)[e]\subset \Fcal_\gamma^G$ is a smooth projective submanifold contained in $Y^o$, $j_X: \Fcal_\gamma\to X$ is the canonical map and $\iota_X: X\croc \Fcal_\gamma^G$ is the inclusion. Finally 
$$
j^*([Y])=j^*_X(\iota_X^*([Y]))=j^*_X\left({\rm Eul}(\Ncal\vert_X)\right)={\rm Eul}(j^{-1}_X(\Ncal\vert_X)),
$$
where $\Ncal$ is the normal bundle of $Y^o$ in $\Fcal^G_\gamma$. 
We leave it to readers to verify that the vector bundle $j^{-1}_X(\Ncal\vert_X)\to \Fcal_\gamma$ corresponds to $\Vcal^{\gamma>0}_\Rbb$. The first point is settled and 
the proof of the second point is similar. 
\end{proof}

\medskip 

We can finally describe the Kirwan polyhedron $\Delta(\T^*\tilde{K}\times V)$.

\begin{theo}\label{theo:delta-exemple-2} Let $K\subset \tilde{K}$ be a closed connected subgroup such as  no non-zero ideal of $\kgot$ is an ideal of $\tilde{\kgot}$. Let $V$ be a $K$-module such as the algebra ${\rm Sym}(V^*)^K$ of invariant polynomial functions on $V$ is reduced to the constants. 

An element $(\tilde{\xi},\xi)\in\tilde{\tgot}_{\geq 0}\times \tgot_{\geq 0}$ belongs to $\Delta(\T^*\tilde{K}\times V)$ if and only if
$$
\langle \tilde{\xi},\tilde{w}\tilde{\gamma}\rangle+\langle \xi,\gamma\rangle\geq 0
$$
for any $(\gamma,\tilde{w})\in\tgot\times \tilde{W}$ satisfying the following properties:
\begin{enumerate}
\item[a)] $\gamma$ is rational and $\dim_T(\qgot^\gamma\times V^\gamma)=1$. 
\item[b)] $[\Xgot_{\gamma}]\cdot \iota^*([\tilde{\Xgot}_{\tilde{w},\gamma}])\cdot {\rm Eul}(V^{\gamma> 0})= [pt]$ in $H^*(\Fcal_\gamma,\Zbb)$.
\item[c)] $\sum_{\stackrel{\alpha\in\Rgot^+}{\langle\alpha,\gamma\rangle> 0}}\langle\alpha,\gamma\rangle=
\sum_{\stackrel{\tilde{\alpha}\in\tilde{\Rgot}^-}{\langle\tilde{\alpha},\tilde{w}\gamma\rangle> 0}}\langle\tilde{\alpha},\tilde{w}\gamma\rangle+ \tr_\gamma(V^{\gamma> 0})$.
\end{enumerate}
\end{theo}

Similar results were obtained by Deltour in his thesis \cite{Deltour-these,Deltour-transf-group}, in the case where $\tilde{K}=K\times K$.


\section{Appendix}

Let $(M,\Omega)$ be a K\"{a}hler Hamiltonian $K$-manifold with proper moment map $\Phi$. The aim of this section is to explain
how we can characterize the subset of analytical semi-stable points in terms of the Kempf-Ness function. 
The result is well-known in the compact case \cite{Teleman04,Mundet10,G-D-S}, and we will explain why it still holds when the moment map is proper.

Let $f:=\frac{1}{2}\|\Phi\|^2$ be the square of the moment map, and let $\nabla f \in {\rm Vect}(M)$ be its gradient vector field. 
We start with the following result  (see Theorem 4.1 in \cite{G-D-S}).

\begin{prop}\label{prop:psi-x-definition}
Fix an element $x\in M$. 
\begin{enumerate}
\item There exists a unique function $\Psi_x:K_\Cbb\to\Rbb$ such as 
\begin{equation}\label{eq:psi-x-definition}
d\Psi_x(g)v:=-\langle \Phi(g^{-1}x),{\rm Im}(g^{-1}v)\rangle,\quad \Psi_x(gk)=\Psi_x(g),\quad \Psi_x(e)=0, 
\end{equation}
for all $g\in K_\Cbb$, $v\in \T_g K_\Cbb$, and all $k\in K$.
\item Define a map $\varphi_x:K_\Cbb\to M$ by $\varphi_x(g)=g^{-1}x$. Then $\varphi_x$ intertwines the gradient vector field $\nabla \Psi_x \in {\rm Vect}(K_\Cbb)$
and the gradient vector field $\nabla f \in {\rm Vect}(M)$ : 
\begin{equation}\label{eq:psi-x-prop}
\T\varphi_x\vert_g\left(\nabla \Psi_x\vert_g\right)=\nabla f\vert_{\varphi_x(g)}.
\end{equation}
\end{enumerate}
\end{prop}

By equation (\ref{eq:psi-x-definition}) we have for all $X\in\kgot$
$$
\Psi_x(e^{-itX})=\int_0^t\langle \Phi(e^{isX}x),X\rangle ds,\quad \forall\, t\geq 0.
$$

Consider the homogeneous space $\mathbb{X}=K_\Cbb/K$ : it is a complete Riemannian manifold with non-positive sectional curvature. 

\begin{defi}
The function $\Psi_x:K_\Cbb\to\Rbb$ is called the lifted Kempf-Ness function based at $x$. It is $K$-invariant and hence descends to a function 
$\Psi_x:\mathbb{X}\to\Rbb$ denoted by the same symbol and called the Kempf-Ness function.
\end{defi}

Now, let us summarize the properties satisfied by the Kempf-Ness functions. The proof given in \cite{G-D-S}[Theorem 4.3] still works when the moment map is proper.

\begin{theo}\label{theo:kempf-ness}
\begin{itemize}
\item The Kempf-Ness function $\Psi_x:\mathbb{X}\to\Rbb$ is Morse-Bott and is convex along geodesics.
\item Every negative gradient flow line $\theta:\Rbb^{\geq 0}\to \mathbb{X}$ of $\Psi_x$ satisfies 
\begin{equation}\label{eq:psi-x-flow}
\lim_{t\to\infty}\Psi_x(\theta(t))=\inf_{\mathbb{X}}\Psi_x.
\end{equation}
\item The Kempf-Ness functions satisfy
\begin{equation}\label{eq:psi-x-h-g}
\Psi_{h^{-1}x}(g)=\Psi_x(hg)-\Psi_x(h),
\end{equation}
for $x\in M$ and $h,g\in K_\Cbb$.
\end{itemize}
\end{theo}

\begin{proof} The first point and the relations (\ref{eq:psi-x-h-g}) are valid without compactness assumption on M. 
We have seen in Proposition \ref{prop:psi-x-definition} that the map $\varphi_x:K_\Cbb\to M$ intertwines the gradient vector fields $\nabla \Psi_x$
and $\nabla f$. As $f$ is proper, every negative gradient flow line of $f$ are well-defined on $\Rbb^{\geq 0}$. Thus every negative gradient flow line of $\Psi_x$ are also 
well-defined on $\Rbb^{\geq 0}$. Then, one checks easily that the rest of the proof given in \cite{G-D-S} works.
\end{proof}

We can now state the generalized Kempf-Ness Theorem (see Theorem 7.3 in \cite{G-D-S}).

\begin{theo}\label{theo:kempf-ness-theorem}
Let $(M,\Omega,\Phi)$ be a K\"{a}hler Hamiltonian $K$-manifold with proper moment map. Let $x\in M$ and denote by $\Psi_x:\mathbb{X}\to\Rbb$ the Kempf-Ness function of $x$. 
Then 
$$
\overline{K_\Cbb\, x}\cap \Phi^{-1}(0)\neq\emptyset\Longleftrightarrow \Psi_x\ {\rm is\ bounded\ below}.
$$
\end{theo}
\begin{proof}
We follow the line of the proof given in \cite{G-D-S}[Theorem 7.3].
Fix a point $x_0\in M$ and denote by $\Psi:\mathbb{X}\to\Rbb$ the Kempf-Ness function of $x_0$. Let $x:\Rbb^{\geq 0}\to M$ be the negative flow line relatively to $f$:
$$
x'=-\mathbb{J}(\Phi(x)\cdot x),\quad {\rm and} \quad x(0)=x_0.
$$
Notice that $x(t)$ is well defined for any $t\geq 0$ because $f:M\to\Rbb$ is proper. Let $g: \Rbb^{\geq 0}\to K_\Cbb$ be the unique solution of the differential equation $g(t)^{-1}g'(t)=i\,\Phi(x(t))$, with initial condition $g(0)=e$. Then $x(t)=g(t)^{-1}x_0$ for all $t\in\Rbb^{\geq 0}$. The limit $x_\infty=\lim_{t\to\infty} x(t)$ exists and 
$\|\Phi(x_\infty)\|=\inf_{K_\Cbb x}\|\Phi\|$ (see Corollary \ref{coro:coro-stratification}).

If $\overline{K_\Cbb\, x}\cap \Phi^{-1}(0)=\emptyset$, then $\|\Phi(x_\infty)\|>0$. Let $\theta:\Rbb^{\geq 0}\to \mathbb{X}$ denotes the composition of $g$ with the projection 
$K_\Cbb\to\mathbb{X}$. By definition of the Kempf-Ness function, we have 
$$
\frac{d}{dt}\left(\Psi\circ \theta\right)=-\langle \Phi(g(t)^{-1}x_0),{\rm Im}(g(t)^{-1}g'(t))\rangle=-\|\Phi(x(t))\|^2\leq -\|\Phi(x_\infty)\|^2 
$$
for all $t\geq 0$. We get then $\Psi(\theta(t))\leq -t\,\|\Phi(x_\infty)\|^2,\, \forall t\geq 0$. Thus $\Psi$ is unbounded below.

Suppose now $\overline{K_\Cbb\, x}\cap \Phi^{-1}(0)\neq \emptyset$. Then $\Phi(x_\infty)=0$. By the Lojasiewicz gradient inequality, we know that there exists 
$t_o>0$, and $a,b >0$ such as $f(x(t))\leq a e^{-bt}$ for $t\geq t_o$. Thus 
$\frac{d}{dt}\left(\Psi\circ \theta\right)=-\|\Phi(x(t))\|^2\geq -2a e^{-bt}$ for $t\geq t_o$. This proves that $\lim_{t\to\infty}\Psi(\theta(t))=\inf_{\mathbb{X}}\Psi$ is finite.
Thus, $\Psi$ is bounded below.
\end{proof}

We state the main result of this appendix which completes the proof of Proposition \ref{prop:N-ss-gamma}. Let us denote by 
$M^{ss}=\{x\in M; \overline{K_\Cbb\, x}\cap \Phi^{-1}(0)\neq\emptyset\}$ the subset of analytical semi-stable points.

\begin{prop}\label{prop:N-ss-gamma-bis}
Let $(\gamma,x)\in\kgot\times M$, such as the limit $x_\gamma:=\lim_{t\to\infty}e^{-it\gamma} x$ exists. 
If $x\in M^{ss}$ and $\langle \Phi(x_\gamma),\gamma\rangle= 0$, then $x_\gamma\in M^{ss}$.
\end{prop}
\begin{proof}Let us suppose that the Kempf-Ness function $\Psi_{x}$ is bounded below by $c\in \Rbb$. Relations (\ref{eq:psi-x-h-g}) give
$$
\Psi_{e^{-it\gamma}x}(g)=\Psi_x(e^{it\gamma}g)-\Psi_x(e^{it\gamma})\geq c +\int_0^t\langle \Phi(e^{-is\gamma}x),\gamma\rangle ds,\quad \forall\, t\geq 0,\quad \forall g\in K_\Cbb.
$$
The function $s\to \langle \Phi(e^{-is\gamma}x),\gamma\rangle$ is decreasing and $\lim_{s\to\infty}\langle \Phi(e^{-is\gamma}x),\gamma\rangle=\langle \Phi(x_\gamma),\gamma\rangle= 0$.
Hence $\langle \Phi(e^{-is\gamma}x),\gamma\rangle\geq 0$ for all $s\geq 0$. We obtain then $\Psi_{e^{-it\gamma}x}(g)\geq c$ for all $g\in K_\Cbb$ and all $t\geq 0$. We get 
finally $\Psi_{x_\gamma}(g)=\lim_{t\to\infty}\Psi_{e^{-it\gamma}x}(g)\geq c$ for all $g\in K_\Cbb$. The Kempf-Ness function $\Psi_{x_\gamma}$ is bounded below, so 
$x_\gamma\in M^{ss}$.
\end{proof}

\bigskip


\end{document}